\documentclass[a4paper,notitlepage,twoside,leqno,10pt]{amsart}

\usepackage{bbm,pifont,latexsym}

\usepackage{dcolumn,indentfirst}
\usepackage[bookmarksopen=true,linktocpage=true,pdfstartview={XYZ null null 1.25}]{hyperref}
\usepackage{amsmath,amssymb,amscd,amsthm,amsfonts,mathrsfs}
\usepackage{color,graphicx,xcolor,graphics,subfigure,extarrows,caption2}
\usepackage{titletoc}

\newtheorem{thm}{Theorem}[section]
\newtheorem{lema}[thm]{Lemma}
\newtheorem{cor}[thm]{Corollary}
\newtheorem{prop}[thm]{Proposition}

\theoremstyle{definition}

\newtheorem*{rmk}{Remark}

\newcommand{\D}{\mathbb{D}}

\newcommand{\R}{\mathbb{R}}

\newcommand{\C}{\mathbb{C}}
\newcommand{\N}{\mathbb{N}}
\newcommand{\EC}{\widehat{\mathbb{C}}}

\newcommand{\MO}{\mathcal{O}}
\newcommand{\MC}{\mathcal{C}}
\newcommand{\MP}{\mathcal{P}}
\newcommand{\MJ}{\mathcal{J}}
\newcommand{\MF}{\mathcal{F}}

\newcommand{\diam}{\textup{diam}}
\newcommand{\dist}{\textup{dist}}
\newcommand{\Shape}{\textup{Shape}}
\newcommand{\Int}{\textup{int}}

\newcommand{\wt}{\widetilde}
\newcommand{\mb}{\mathbb}
\newcommand{\mc}{\mathcal}
\newcommand{\sm}{\setminus}
\newcommand{\tu}{\textup}
\newcommand{\ol}{\overline}
\newcommand{\es}{\emptyset}

\newcommand{\tb}{\textbf}

\makeatletter\@addtoreset{equation}{section}\makeatother 

\begin{document}

\author[W. Qiu]{Weiyuan Qiu}
\address{School of Mathematical Sciences, Fudan University, Shanghai 200433, P. R. China}
\email{wyqiu@fudan.edu.cn}

\author[F. YANG]{FEI YANG}
\address{Department of Mathematics, Nanjing University, Nanjing 210093, P. R. China}
\email{yangfei@nju.edu.cn}

\author[J. Zeng]{Jinsong Zeng$^\dag$}
\address{School of Mathematics and Information Science, Guangzhou University, Guangzhou 510006, P. R. China}
\email{zeng.jinsong@foxmail.com}

\title[Quasisymmetric geometry of carpet Julia sets]{Quasisymmetric geometry of Sierpi\'{n}ski carpet Julia sets}

\begin{abstract}
In this paper, the main focus is on the Sierpi\'{n}ski carpet Julia sets of the rational maps with non-recurrent critical points. We study the uniform quasicircle property of the peripheral circles, the relatively separated property of the peripheral circles and the locally porous property of these carpets. We also establish some quasisymmetric rigidities of these carpets, which generalizes the main results of Bonk-Lyubich-Merenkov \cite{BLM16} to the postcritically infinite case. In the end we give a strategy to construct a class of postcritically infinite rational maps whose Julia sets are quasisymetrically equivalent to some round carpets.
\end{abstract}

\subjclass[2010]{Primary: 37F45; Secondary: 37F10}

\keywords{Sierpi\'{n}ski carpet; Julia sets; quasisymmetric rigidity; semi-hyperbolic}

\date{\today}

\thanks{$^\dag$ Corresponding author}


\maketitle

\tableofcontents

\section{Introduction}\label{introduction}

\subsection{Quasisymmetric equivalence}

Let $(X,d_X)$ and $(Y,d_Y)$ be two metric spaces. If there exists a homeomorphism $f:X\rightarrow Y$ and a distortion control function $\eta:[0,\infty)\rightarrow [0,\infty)$ which is also a homeomorphism such that
\begin{equation*}
\frac{d_Y(f(x),f(y))}{d_Y(f(x),f(z))}\leq \eta\left(\frac{d_X(x,y)}{d_X(x,z)}\right)
\end{equation*}
for every triple of distinct points $x,y,z\in X$, then $f$ is called a \emph{quasisymmetric map} and $(X,d_X)$, $(Y,d_Y)$ are \emph{quasisymetrically equivalent}.  A basic question in quasiconformal geometry is to determine whether two given homeomorphic metric spaces are quasisymmetrically equivalent.

It is known that the question arises also in the classification of hyperbolic spaces and word hyperbolic groups in the sense of Gromov \cite{BP02,Kle06}. See also \cite{Bou97} for the examples of inequivalent spaces modeled on the universal Menger curve.

According to \cite{Why58}, a set $S\subseteq\EC$ is called a \emph{Sierpi\'{n}ski carpet} (\emph{carpet} in short) if $S$ has empty interior and can be expressed as $S=\EC\setminus\bigcup_{i\in\mathbb{N}}D_i$, where $\{D_i\}_{i\in\mathbb{N}}$ are Jordan disks with their closures pairwise disjoint and with their spherical diameters $\textup{diam}_{\sigma}(D_i)\to 0$ as $i\to\infty$. The collection of the boundaries of the Jordan disks $\{\partial D_i\}_{i\in\N}$ are called the \emph{peripheral circles} of $S$. If each peripheral circle $\partial D_i$ is geometrically round, then $S$ is called a \emph{round carpet}. All Sierpi\'{n}ski carpets are topologically the same, so the question on classification of carpets up to quasisymmetric equivalence arises naturally.

Actually, the study of the quasisymmetric equivalences between the carpets and round carpets was partially motivated by the Kapovich-Kleiner conjecture in geometry group theory. This conjecture is equivalent to the following statement: if the boundary of infinity $\partial_\infty G$ of a Gromov hyperbolic group $G$ is a Sierpi\'{n}ski carpet, then $\partial_\infty G$ is quasisymmetrically equivalent to a round carpet in $\EC$.

\vskip0.2cm
The example of rational map whose Julia set is a Sierpi\'{n}ski carpet, called \emph{carpet Julia set} for short, was firstly found by Milnor and Tan \cite[Appendix F]{Mil93}. Later, it was proved that many rational maps, such as the McMullen maps \cite{DLU05}, the generalized McMullen maps \cite{XQY14}, the quadratic rational maps \cite{DFGJ14} and some higher degree rational maps \cite{Yan18} etc, have carpet Julia sets. Conjecturally, the hyperbolic components of such rational maps are relatively compact in the space of rational maps up to M\"{o}bius conjugation \cite[Question 5.3]{McM95}.

\vskip0.2cm
Let $f$ be a rational map whose Julia set $\MJ_f$ is a carpet. Two questions arise naturally:

(Q1) Whether there exist carpet Julia sets such that they are not quasisymmetrically equivalent to $\mc{J}_f$ ?

(Q2) When is $\MJ_f$ quasisymmetrically equivalent to a round carpet?

\vskip0.2cm
Let $X$ be a metric space. The \textit{conformal dimension} of $X$ is the infimum of the Hausdorff dimensions of all metric spaces which are quasisymmetrically equivalent to space $X$. The conformal dimension is invariant under quasisymmetric maps. For the first question, Ha\"{i}ssinsky and Pilgrim constructed a sequence of hyperbolic rational maps with carpet Julia sets, whose degrees tend to infinity,  and showed that their conformal dimensions tend to two \cite[Theorem 3]{HP12}. This means that there are infinitely many quasisymmetrically inequivalent carpet Julia sets. However, it is of interest to ask whether there exists such sequence of rational maps with uniformly bounded degrees.

Let $\textup{dist}_\sigma$ and $\textup{diam}_\sigma$ denote the spherical distance and diameter respectively. The \emph{relative distance} $\Delta(A,B)$ of two sets $A$ and $B$ in $\EC$ is defined as
\begin{equation}\label{defi-relative-distance}
\Delta(A,B):=\frac{\textup{dist}_\sigma(A, B)}{\textup{min}\{\textup{diam}_\sigma(A), \textup{diam}_\sigma(B)\}}.
\end{equation}

A set of Jordan curves $\mathcal{C}=\{\gamma_i\}_{i\in\N}$ is called \emph{uniformly relatively separated} if their pairwise relative distances are uniformly bounded away from zero. Specifically, there exists $\delta>0$ such that $\Delta(\gamma_i, \gamma_j)\geq\delta$ for every two different $i$ and $j$. The set $\mathcal{C}$ is \emph{uniform quasicircles} if there exists $K\geq 1$ such that each $\gamma_i$ in $\mathcal{C}$ is a $K$-quasicircle.

For the question (Q2), Bonk gave a sufficient condition on the carpets in $\EC$ such that they can be quasisymmetrically equivalent to some round carpets. He proved that a carpet $S$ in $\EC$ is quasisymmetrically equivalent to a round carpet if its peripheral circles is uniform quasicircles and is uniformly relatively separated \cite[Corollary 1.2]{Bon11}. It is worth to mention that quasisymmetric maps preserve the uniform quasicircles and uniformly relatively separated properties. However, it is not hard to see that the peripheral circles of a carpet that is quasisymmetrically equivalent to a round carpet must be uniform quasicircles but is not necessarily uniformly relatively separated.

\vskip0.2cm
Bonk, Lyubich and Merenkov studied the postcritically-finite rational maps whose Julia sets are Sierpi\'{n}ski carpets. They proved that such carpet Julia sets are quasisymmetrically equivalent to some round carpets \cite[Theorem 1.10]{BLM16}. Moreover, they show that any quasisymmetric map between two critical finite carpet Julia sets is the restriction of a M\"{o}bius transformation \cite[Theorem 1.4]{BLM16}. As a corollary, the quasisymmetric group, consisting of all quasisymmetric self-maps, is finite \cite[Corollary 1.2]{BLM16}.

In this article, we study the carpet Julia sets in the postcritically-infinite case and extend the corresponding results to a more general case.

\subsection{Statement of the main results}

The \emph{$\omega$-limit set} $\omega(x)$ of a point $x\in\EC$ under a rational map $f$ is defined as the set of accumulation points in the orbit of $x$. More precisely, $\omega(x):=\{y\in\EC:$ there exists a sequence $\{k_n\}_{n\in\N}$ such that $\lim_{n\to\infty}f^{\circ k_n}(x)=y\}$. Obviously, $\omega(x)$ is forward invariant under $f$. For a given rational map $f$, we use $\mathcal{C}_f$ to denote the family of the boundaries of the Fatou components of $f$.
We establish a sufficient condition on the carpet Julia sets such that they are quasisymmetrically equivalent to some round carpets.

\begin{thm}\label{thm-main-1}
	Let $f$ be a rational map whose Julia set $\mathcal{J}_f$ is a carpet. If the elements in $\mathcal{C}_f$ are disjoint from the $\omega$-limit sets of the critical points, then $\mathcal{C}_f$ are uniform quasicircles and uniformly relatively separated. In particular, $\mathcal{J}_f$ is quasisymmetrically equivalent to a round carpet.
\end{thm}

A critical point $c$ of $f$ is called \emph{recurrent} if $c\in\omega(c)$. A rational map $f$ is called \emph{semi-hyperbolic} provided that the Julia set $\mathcal{J}_f$ contains neither parabolic periodic points nor recurrent critical points (see \cite{Man93} and \cite{Yin99}). It was known that the Julia set of a semi-hyperbolic rational map is locally connected and has measure zero or equal to $\EC$.

\begin{thm}\label{thm-main-2}
	Let $f$ be a semi-hyperbolic rational map whose Julia set $\mathcal{J}_f$ is a carpet. Then $\mathcal{C}_f$ are uniform quasicircles. Moreover, they are uniformly relatively separated if and only if the $\omega$-limit sets of the critical points are disjoint from the elements of $\mathcal{C}_f$.
\end{thm}

If a rational map is not semi-hyperbolic, then the boundary of some Fatou component may not be a quasicircle although it is a Jordan curve. For example, one can construct a rational map $f$ whose Julia set is a Sierpi\'{n}ski carpet but the Julia set $\MJ_f$ contains a parabolic periodic point. The corresponding parabolic Fatou component contains exactly one petal and has infinitely many cusps on its boundary. Thus the boundary of this Fatou component cannot be a quasicircle. In this case, $\mathcal{J}_f$ cannot be quasisymmetrically equivalent to any round carpet.

As an application of Theorem \ref{thm-main-2}, we prove

\begin{thm}\label{thm-crit-infi}
There exists a rational map which is critically-infinite in the Julia set such that the corresponding Julia set is quasisymmetrically equivalent to a round carpet.
\end{thm}

One can also refer to \cite{QYY16} for other non-hyperbolic examples whose Julia sets are quasisymmetrically equivalent to round carpets.

\vskip0.2cm
For a given rational map $f$, we use $\MF_f$ and $\MP_f$ to denote its Fatou set and the postcritical set respectively.
As a generalization of the main result in \cite{BLM16}, we prove the following

\begin{thm}\label{mobius}
	Let $f,g$ be two rational maps whose Julia sets are carpets. Suppose that $f$ is semi-hyperbolic, and that the $\omega$-limit sets of the critical points of $f$ are disjoint from the elements of $\mathcal{C}_f$, and that $\sharp(\MF_f\cap\MP_f)<\infty$, $\sharp(\MF_g\cap\MP_g)<\infty$. Then any quasisymmetric homeomorphism between $\mathcal{J}_f$ and $\mathcal{J}_g$ is the restriction of a M\"{o}bius transformation.
\end{thm}

It is worth to mention that $g$ is not assumed to be semi-hyperbolic in advance. The proof of Theorem \ref{qs_rigidity} indicates that $g$ has no Siegel disks. We conjecture that $g$ turns out to be semi-hyperbolic as well.

If $f$ is a hyperbolic rational map that is not conjugate to $z^d$ with $|d|\geq 2$, and whose Julia set is not a circle, an arc of a circle nor the whole Riemann sphere, then the group of M\"{o}bius transformation that keeps  $\MJ_f$ invariant is finite (see \cite{Lev90,Lev01} and also \cite{LP97}). It was proved in \cite{BLM16} that the quasisymmetric group of the carpet Julia set of a postcritically-finite rational map is finite. We now extend this result to a more general case.

\begin{thm}\label{group_finite}
	Let $f$ be a semi-hyperbolic rational map whose Julia set $\mathcal{J}_f$ is a carpet. If the $\omega$-limit sets of the critical points of $f$ are disjoint from the elements of $\mathcal{C}_f$, then the quasisymmetric group $QS(\mathcal{J}_f)$ is finite.
\end{thm}

\subsection{Outline of the proofs and the organization of this article}

In the first half part of this article, we are mainly interested in the condition when a carpet Julia set is quasisymmetrically equivalent to a round carpet. By Bonk's criterion, this motivates us to find the condition when the peripheral circles of a carpet Julia set are uniform quasicircles and when they are uniformly relatively separated.

In order to prove that the peripheral circles of some carpet Julia sets are uniform quasicircles, we first discuss the periodic Fatou components and prove that they are quasidisks if their boundaries avoid the parabolic periodic points and the points in the $\omega$-limit sets of the recurrent critical points (Lemma \ref{periodic uniform quasicircles}). Therefore, all peripheral circles are quasicircles by using Sullivan's eventually periodic theorem.

In order to prove the uniformity, we discuss two cases. In the first case, suppose that all the periodic Fatou components are disjoint from the $\omega$-limit sets of the critical points. Then for each periodic Fatou component $U$, one can find a large Jordan disk $V$ such that $V\setminus\overline{U}$ is an annulus and all components of the preimages of $V\setminus\overline{U}$ are annuli whose moduli have uniform lower bound. By using a distortion argument, one can prove that all peripheral circles are uniform quasicircles (Proposition \ref{uniform_quasicircle-1}). In the second case, suppose that the rational map is semi-hyperbolic. Then the corresponding Julia set (and hence all the periodic Fatou components) contains neither parabolic periodic points nor recurrent critical points. One can also prove that all peripheral circles are uniform quasicircles by using Ma\~{n}\'{e}'s theorem and its variation (Theorem \ref{mane lemma}, Lemma \ref{remk} and Proposition \ref{uniform_quasicircle}).

In order to prove that the peripheral circles of some carpet Julia sets are uniformly relatively separated, we first establish a lemma which asserts that the modulus can control the relative distance (Lemma \ref{modulus and separated}). Then we prove the peripheral circles are uniformly relatively separated by showing that all moduli of the annuli between two different peripheral circles have a lower positive bound (Proposition \ref{separated}).

\vskip0.2cm
In the second half part of this article, we study the quasisymmetric rigidity of the carpet Julia sets of semi-hyperbolic rational maps. For this, we need to prepare two important ingredients: the first one is to establish the locally porous property of the Julia sets of semi-hyperbolic rational maps (Theorem \ref{local_porous}), and the second one is the dynamics on the Fatou components of the rational maps whose Julia set is locally connected (\S\ref{sec-dyn-Fatou}). With these two ingredients in hand, we can prove two stronger results (Theorems \ref{qs_rigidity} and \ref{finite_group}) such that Theorems \ref{mobius} and \ref{group_finite} are their corollaries by combining some conclusions obtained in \cite{Mer12, Mer14} and making a detail analysis on the classification of the M\"{o}bius transformations.

\vskip0.2cm
This article is organized as follows:

In $\S$\ref{sec-distortion-esti}, we prepare some distortion lemmas for the proofs of Theorems $\ref{thm-main-1}$ and \ref{thm-main-2}. In particular, we prove that the modulus can control the relative distance.

In $\S$\ref{sec-pf-of-main}, we first prove some propositions about the properties of uniform quasicircles and uniformly relatively separated. Then we prove Theorem \ref{thm-main-1} by using Bonk's criterion. We also discuss the non-uniformly relatively separated property and prove Theorem \ref{thm-main-2} by combining Bonk's criterion and Ma\~{n}\'{e}-Yin's characterization on semi-hyperbolic rational maps. We will finish this section by proving that a semi-hyperbolic rational map with carpet Julia set is locally porous.

In \S\ref{sec-dyn-Fatou}, we give a characterization of the dynamics on the Fatou components of a rational map whose Julia set is locally connected and the intersection of the postcritical set with the Fatou set is finite.

In \S\ref{quas_rigidity}, in order to obtain Theorems \ref{mobius} and \ref{group_finite}, we will prove two stronger results (Theorems \ref{qs_rigidity} and \ref{finite_group}).

In \S\ref{example}, using the combinatorial method and renormalization theory, we construct a critically-infinite semi-hyperbolic rational map whose Julia set is quasisymmetrically equivalent to a round carpet and hence prove Theorem \ref{thm-crit-infi} (see Theorem \ref{existence-mcm}).

\subsection{Notations}\label{notations}

We list here some notations that will be used throughout in this article.

$\bullet$ ${\mathbb{C}}$, $\EC$ and $\mathbb{D}$ are the complex plane, the Riemann sphere and the unit disk respectively.

$\bullet$ $\overline{A}$, $\partial A$ and $\Int(A)$ are the closure, the boundary and the interior of a set $A$ respectively.

$\bullet$ Two sets satisfying $A\Subset B$ means that $\overline{A}\subseteq \text{int}(B)$.

$\bullet$ Let $z\in\C$, $E\subseteq\C$ and $r>0$. We denote $B(z,r):=\{x\in\C; |x-z|<r\}$ and $\text{diam}(E):=\text{sup}\{\text{dist}(x,y); x,y\in E\}$, where $\text{dist}(x,y):=|x-y|$ is the Euclidean metric.

$\bullet$ Let $z\in\EC$, $E\subseteq\EC$ and $r>0$. We denote $B_{\sigma}(z,r):=\{x\in\EC;\text{dist}_\sigma(x,z)<r\}$ and $\text{diam}_\sigma(E):=\text{sup}\{\text{dist}_{\sigma}(x,y); x,y\in E\}$, where $\text{dist}_\sigma(,)$ is the spherical metric.

$\bullet$ The set of critical points of a branched covering map $f$ is denoted by $\text{Crit}(f)$ and by $\MP_f$ the set of postcritical points.

$\bullet$ The Julia set and the Fatou set of a rational map $f$ are denoted by $\mathcal{J}_f$ and $\mathcal{F}_f$ respectively.

$\bullet$ The family of Fatou components is denoted by $\text{Comp}(\mathcal{F}_f)$ and the family of boundaries of Fatou components is $\mathcal{C}_f$.

\section{Some distortion estimations}\label{sec-distortion-esti}

In this section, we give some distortion estimations and useful lemmas, which will be used in the following sections.

\subsection{The modulus controls the relative distance}

Let $A$ be an annulus with non-degenerated boundary components. Then there exists a conformal map sending $A$ to a standard annulus $\{z\in\C:0<r<|z|<1\}$, where $r>0$ is uniquely determined by $A$. The \textit{modulus} of $A$ is defined as $\textup{mod}(A)=\frac{1}{2\pi}\log(1/r)$, which is invariant under conformal maps.

Recall that the \emph{relative distance} $\Delta(A,B)$ of two subsets $A$ and $B$ in $\EC$ is defined in \eqref{defi-relative-distance}. Now we prove that relative distance of two disjoint Jordan curves can be controlled by the modulus of the annulus between them.

\begin{lema}[{Modulus controls the relative distance}]\label{modulus and separated}
	Let $A\subseteq\EC$ be an annulus with two boundary components $\gamma_1$ and $\gamma_2$. If the modulus of $A$ satisfies $\textup{mod}(A)\geq m>0$, then there exists a constant $C(m)>0$ depending only on $m$ such that the relative distance of $\gamma_1$ and $\gamma_2$ satisfies $\Delta(\gamma_1, \gamma_2)\geq C(m)>0$.
\end{lema}

\begin{proof}
	Without loss of generality, we assume that $A\subseteq \C$, $\gamma_1$, $\gamma_2$ are not singletons and $0<\diam(\gamma_1)\leq\diam(\gamma_2)$ and
	\begin{equation}\label{est-0}
	\dist(\gamma_1,\gamma_2)=|x-y|
	\end{equation}
	for $x\in \gamma_1$ and $y\in \gamma_2$. There exists a point $z\neq x$ in $\gamma_1$ such that $|x-z|=\mathop{\textup{sup}}^{}_{a\in \gamma_1}|a-x|$. Therefore, we have
	\begin{equation}\label{est-1}
	\diam(\gamma_1)\leq 2|x-z|.
	\end{equation}
	
	Consider the linear function $h(t)=(t-x)/(x-z)$, which maps $x,y,z$ to $0,(y-x)/(x-z)$ and $-1$. Then $h(A)$ is an annulus separating the points $0$ and $-1$ from $h(y)$ and $\infty$. Let
	\begin{equation*}
	R=|h(y)|=|(y-x)/(x-z)|.
	\end{equation*}
	By Teichm\"{u}ller's Module Theorem (see for example, \cite[p.\,56]{LV73}), we have
	\begin{equation*}
	m\leq \textup{mod}(A)=\textup{mod} (h(A))\leq 2\,\mu\left(\sqrt{\frac{1}{1+R}}\right),
	\end{equation*}
	where $r\mapsto\mu(r)$ is a continuous and strictly decreasing map defined on the interval $(0,1)$. By \eqref{est-0} and \eqref{est-1}, this means that the relative distance of $\gamma_1$ and $\gamma_2$ is
	\begin{equation*}
	\Delta(\gamma_1, \gamma_2)=\frac{\textup{dist}(\gamma_1,\gamma_2)}{\textup{diam}(\gamma_1)}\geq \frac{|x-y|}{2|x-z|}= \frac{R}{2}\geq \frac{1}{2}\left(\frac{1}{(\mu^{-1}(m/2))^2}-1\right):=C(m).
	\end{equation*}
	The proof is complete.
\end{proof}

\subsection{Distortion on the shape and turning}

Let $U$ be a proper subset of $\C$ and $z\in \text{int}(U)$. The \emph{shape} of $U$ about $z$, denoted by $\Shape(U, z)$, is defined as
\begin{equation*}
\Shape (U,z)=\frac{\sup_{w\in\partial U}|w-z|}{\inf_{w\in\partial U}|w-z|}=\frac{\sup_{w\in\partial U}|w-z|}{\dist(z,\partial U)}.
\end{equation*}
It is obvious that $\Shape(U, z)=\infty$ if and only if $U$ is unbounded and $\Shape(U, z)=1$ if and only if $U$ is a round disk centered at $z$.
In other cases, $1 < \Shape(U, z) < \infty$. Note that for any domain $\Omega$ bounded by a $K$-quasicircle, there exist a point $z\in\Omega$ and a constant $C(K)\geq 1$ such that $\text{Shape}(\Omega, z)\leq C(K)$.

Let $E$ be a non-degenerated continuum (non-singleton, connected and compact) in $\mathbb{C}$. For any $z_1, z_2 \in E$, the \emph{turning} of $E$ about $z_1$ and $z_2$ is defined by
\begin{equation*}
\Lambda(E; z_1, z_2) = \frac{\diam(E)}{|z_1 - z_2|}.
\end{equation*}
It is easy to see that $1\leq \Lambda(E; z_1, z_2) \leq \infty$ and $\Lambda(E; z_1, z_2) = \infty$ if and only if $z_1=z_2$.

\begin{lema}[{Distortion of the shape and turning when pullback, \cite[Lemma 6.1]{QWY12}}]\label{control turning}
	Let $U_i\Subset V_i\neq \C$ be a pair of Jordan disks with $\textup{mod}(V_2\setminus\overline{U}_2) \geq m>0$, where $i=1,2$.
	Suppose that $f : V_1 \to V_2$ is a proper holomorphic map of degree $d\geq 1$ and $U_1$ is a component of $f^{-1}(U_2)$. Then there are two positive constants $C_1(d,m)$ and $C_2(d,m)$ depending only on $d$ and $m$, such that
	
	$(1)$ For all $z\in U_1$, the shape satisfies
	\begin{equation*}
	\textup{Shape}(U_1,z)\leq C_1(d,m)\ \textup{Shape}(U_2,f(z)).
	\end{equation*}
	
	$(2)$ For any connected and compact subset $E$ of $U_1$ with the cardinal number $\sharp E \geq 2$ and any $z_1, z_2 \in E$, the turning satisfies
	\begin{equation*}
	\Lambda(E; z_1,z_2)\leq C_2(d,m)\,\Lambda(f(E); f(z_1),f(z_2)).
	\end{equation*}
\end{lema}

The lemma stated above means that the shape and the turning of the interior boundary of an annulus can be controlled under a proper holomorphic map if the modulus of this annulus has a lower bound.

By definition (see for example, \cite[p.\,100]{LV73}), a Jordan curve $\gamma$ is called a \textit{quasicircle} if there exists a positive constant $K\geq 1$ such that for any different points $x,y\in \gamma$, the turning of $\gamma$ about $x$ and $y$ satisfies
\begin{equation*}
\Lambda(\gamma';x,y)\leq K,
\end{equation*}
where $\gamma'$ is one of the two components of $\gamma\setminus\{x,y\}$ with smaller diameter.
By Lemma \ref{control turning} (2), we have the following immediate corollary.

\begin{cor}\label{pullback of quasicircle}
	Let $U_i\Subset V_i\neq \C$ be a pair of Jordan disks, where $i=1,2$. Suppose that $\textup{mod}(V_2\setminus\overline{U}_2) \geq m>0$ and $f: V_1\to V_2$ is a conformal map with $f(U_1)=U_2$. If $\partial U_2$ is a $K$-quasicircle, then there is a constant $C(K,m)\geq 1$ such that $\partial U_1$ is a $C(K,m)$-quasicircle.
\end{cor}

When one pushforwards a topological disk, the image may be complicated. We do not even know whether it is simply connected or not. However, the following Lemma gives us a control of the shape.

\begin{lema}[{Distortion of shape when pushforward, \cite[Corollary 2.3]{Yin00}}]\label{forward_dis}
	Let $U\ni z_0$ be a simply connected domain in $\C$ and $f:U\to B(w_0, 2\delta)$ a proper holomorphic map of $\deg f\leq d$ with $w_0=f(z_0)$. If $f$ maps $B(z_0,r)\subseteq U$ into $B(w_0,\delta)$, then there exists a constant $K$ depending only on $d$ so that $\Shape(f(B(z_0,r)),w_0)\leq K$.
\end{lema}

We will use the following estimations on the modulus in the future.

\begin{lema}[{\cite[Lemma 4.5]{KL09}}]\label{lem:KL09}
	Let $U_i\Subset V_i\neq\C$ be a pair of Jordan disks, where $i=1,2$. Suppose that $f:V_1\rightarrow V_2$ is a proper holomorphic map of degree $d\geq 1$ and $U_1$ is a component of $f^{-1}(U_2)$. Then
	$$\textup{mod}(V_1\setminus\overline{U}_1)\leq\textup{mod}(V_2\setminus\overline{U}_2)\leq d\ \textup{mod}(V_1\setminus\overline{U}_1).$$
\end{lema}

\subsection{Ma\~{n}\'{e}'s theorem and a complement}

In this subsection, we first recall a theorem due to Ma\~{n}\'{e} and then give a complementary lemma under the same condition of Ma\~{n}\'{e}'s theorem. They will be used frequently later.

\begin{thm}[{\cite[Theorem II]{Man93}}]\label{mane lemma}
	Let $f:\EC\to\EC$ be a rational map with degree at least two. If a point $x\in \mathcal{J}_f$ is not a parabolic periodic point and is not contained in the $\omega$-limit set of a recurrent critical point, then for any $\epsilon>0$ there exists an open neighborhood $U_x$ of $x$ such that:
	
	\textup{(P1)} For all $n\geq 0$, every component of $f^{-n}(U_x)$ has diameter $\leq\epsilon$;
	
	\textup{(P2)} There exists $d>0$ such that for all $n\geq0$ and every connected component $V$ of $f^{-n}(U_x)$, the degree of $f^{\circ n}:V\to U_x$ is $\leq d$.
\end{thm}

When we pull back a Jordan disk $U$ by a rational map $f$, there maybe exist a component $W$ of $f^{-1}(U)$ which is not simply connected.
If the boundary $\partial U$ avoids the critical values, then $\partial W$ is the union of finitely many disjoint Jordan curves $\{C_i\}$. Moreover, we have $f(C_i)=\partial U$ for each $i$. Note that $W$ is a connected set whose boundary consists of finitely many Jordan curves. We have $\EC\setminus \overline{W}=\bigcup_i V_i$, where each $V_i$ is a Jordan disk bounded by the Jordan curve $C_i$. Since the restriction of $f$ on $V_i$ is a holomorphic branched covering and $f(\partial V_i)=\partial U$, we have $f(V_i)=\EC$ or $f(V_i)=\EC\setminus\overline{U}$. In other words, the image of each component of the complement of $\overline{W}$ under $f$ is either $\EC$ or $\EC\setminus\overline{U}$. See Figure \ref{simply_connect} for an example.

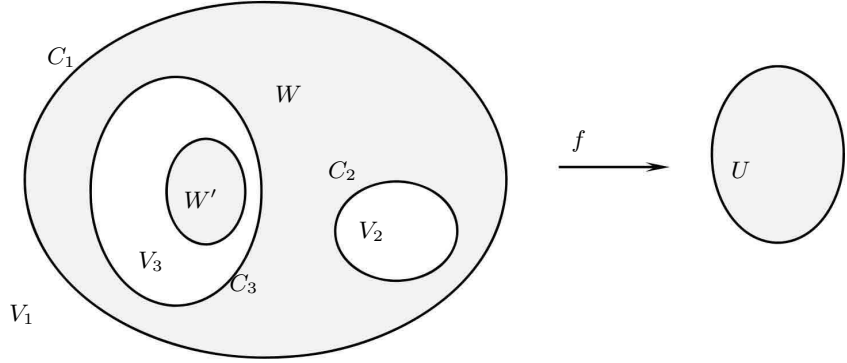
\begin{figure}[!htpb]
	\setlength{\unitlength}{1mm}
	\centering
	\includegraphics[width=110mm]{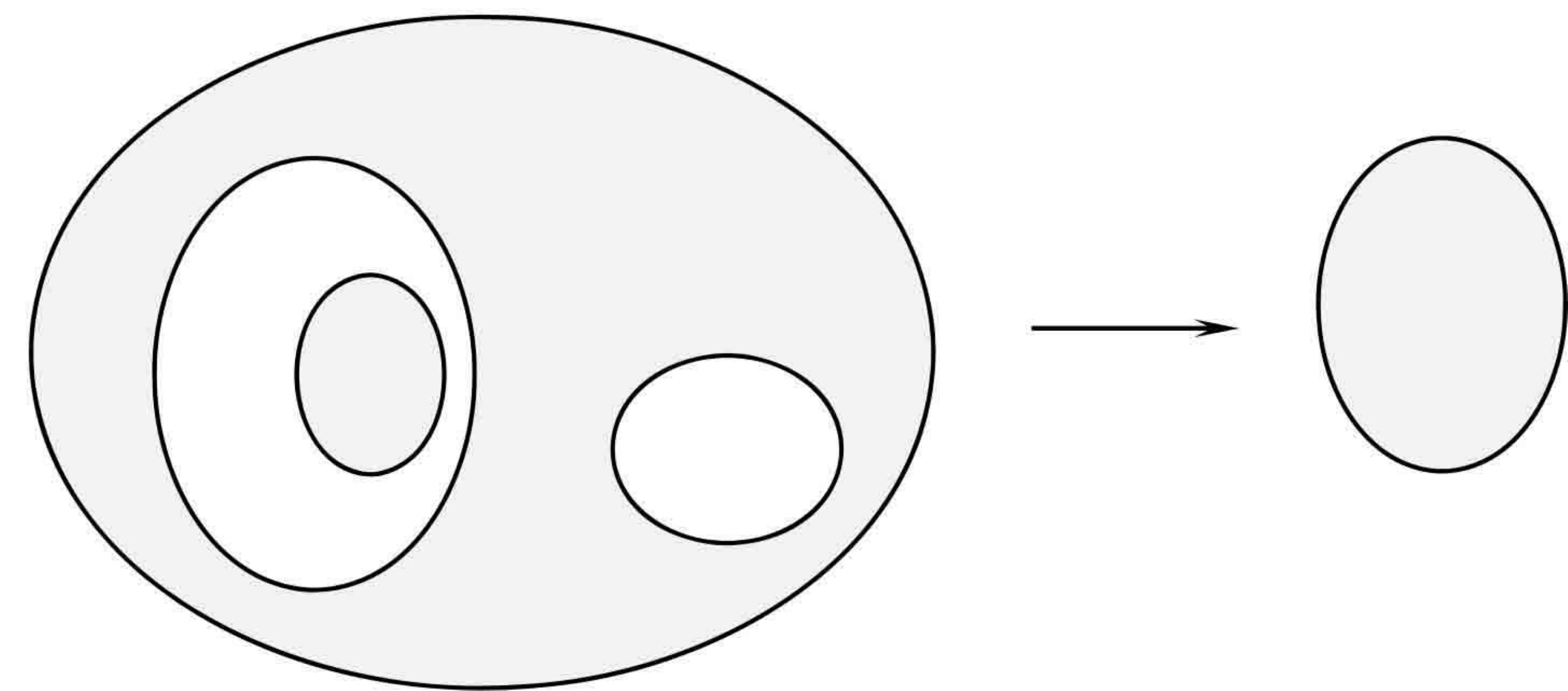}
	\small
	\put(-15,24){$U$}
	\put(-36,28){$f$}
	\put(-64,16){$V_2$}
	\put(-93,12){$V_3$}
	\put(-110,5){$V_1$}
	\put(-75,34){$W$}
	\put(-87,20){$W'$}
	\put(-81,9){$C_3$}
	\put(-68,24){$C_2$}
	\put(-105,39){$C_1$}
	\caption{The pull back of a simply connected domain $U$ under the rational map $f$ with degree $4$, where $f(W)=U$ and $\partial W=C_1\cup C_2\cup C_3$. The complement of $\overline{W}$ consists of $3$ simply connected components $V_1$, $V_2$ and $V_3$. In particular, $f(V_1)=f(V_2)=\EC\setminus\overline{U}$ and $f(V_3)=\EC$. Moreover, $f^{-1}(U)=W\cup W'$, $W$ contains $4$ critical points of $f$ and $V_3\setminus W'$ (the white annulus) contains two.}
	\label{simply_connect}
\end{figure}

In the rest of this article, we only consider rational maps whose Julia sets are not the whole complex sphere. Therefore, after conjugating $f$ by a suitable M\"obius transformation, we always assume that $\infty$ lies in the Fatou set. This means that $\MJ_f$ is a compact set in $\C$.

\begin{lema}\label{remk}
	Let $f$ be a rational map with degree at least two and $\MJ_f\subseteq\C$. Suppose that $x\in \MJ_f$ is neither a parabolic periodic point nor contained in the $\omega$-limit set of a recurrent critical point. Then there is an open neighborhood $U_x$ of $x$ such that
	
	\textup{(P3)} For all $n\geq 0$, every connected component of $f^{-n}(U_x)$ is simply connected.
\end{lema}

\begin{proof}
	By the assumption that $\infty\not\in \mathcal{J}_f$, the grand orbit of $\infty$ lies in the Fatou set of $f$. Let $\delta_0>0$ be a small positive number such that
	\begin{equation}\label{delta0}
	0<\delta_0\leq\textup{dist}(f^{-1}(\infty),\mathcal{J}_f)/2.
	\end{equation}
	By Theorem \ref{mane lemma}, there exists an open neighborhood $U_x'$ of $x$ such that every component of $f^{-n}(U_x')$ has diameter $\leq\delta_0$ for all $n\geq 0$.
	
	Let $U_x:=B(x,\delta_x)$ be the largest round disk which is contained in $U_x'$. We claim that every component $W_n$ of $f^{-n}(U_x)$ is simply connected. If not, let $V_n$ be a bounded component of $\mathbb{C}\setminus \overline{W}_n$, where $n\geq 1$. Then $\partial V_n\subseteq\partial W_n$ and so $\textup{diam}(V_n)\leq\delta_0$. This means that $V_n$ cannot intersect $f^{-1}(\infty)$. Inductively, one can easily check that $f^{\circ k}(V_n)\cap f^{-1}(\infty)=\emptyset$ for all $0\leq k\leq n-1$. It follows that $\infty\not\in f^{\circ n}(V_n)$, which is a contradiction since  $f^{\circ n}(V_n)=\EC$ or $f^{\circ n}(V_n)=\EC\setminus\overline{U}_x$. Therefore, such $V_n$ does not exist. This means that $W_n$ is simply connected.
\end{proof}

Lemma \ref{remk} is useful in the following since we need to obtain the simply connected preimages of a simply connected domain.

\section{Geometry of the boundaries of Fatou components}\label{sec-pf-of-main}

If a rational map $f$ whose Julia set $\mathcal{J}_f$ is a carpet, then $f$ cannot be a polynomial. In fact,
the intersection of the closures of the bounded Fatou components (if any) and the basin of infinity of $f$ is non-empty provided $f$ is a polynomial since the Julia set $\MJ_f$ is the boundary of the basin of infinity. In order to prove Theorem \ref{thm-main-1}, we need to show that the peripheral circles of the carpets are uniform quasicircles and uniformly relatively separated by Bonk's criterion.

\subsection{Sufficiency for the property of uniform quasicircles}

In this subsection, we prepare some lemmas and give two sufficient conditions such that the boundaries of the Fatou components are uniform quasicircles. We first discuss the regularity of the boundaries of the periodic Fatou components and then spread the results to their all preimages.

\begin{lema}\label{jordan_curve}
	Let $\Gamma$ be a Jordan curve in the plane $\mathbb{C}$. Then there exists a constant $\delta_\Gamma>0$ depending only on $\Gamma$ such that, for any Jordan subarc $\gamma\subseteq\Gamma$ with $\diam(\gamma)\leq \delta_\Gamma$, one has $\textup{diam}(\gamma)<\textup{diam}(\Gamma\setminus\gamma)$.
\end{lema}

\begin{proof}
	Consider the function $h: \Gamma\times\Gamma\to\mathbb{R}$, $(x,y)\mapsto\textup{diam}(L(x,y))$, where $L(x,y)$ is one of the two components of $\Gamma\setminus\{x,y\}$ with larger diameter (we define $h(x,y):=\diam(\Gamma)$ if $x=y$). Obviously, the map $h$ is continuous. Since $\Gamma\times\Gamma$ is compact, the function $h$ has a minimum $\delta>0$. Then the lemma holds if we set $\delta_\Gamma=\delta/2$.
\end{proof}

\begin{lema}\label{local_homeomorphic}
	Let $f$ be a rational map whose Fatou components are Jordan domains. Then for any Fatou component $U$, the restriction $f|_{\partial U}$ is locally one to one.
\end{lema}

\begin{proof}
	If not, there exist a critical point $z\in\partial U$ and two arbitrary small closed subarcs $\gamma_1,\gamma_2$ of $\partial U$ with $\gamma_1\cap\gamma_2=\{z\}$ such that $f(\gamma_1)=f(\gamma_2)$. In the neighborhood of $z$, we choose open arcs $\alpha_1$ resp. $\alpha_2$ in $U$ connecting the endpoints of $\gamma_1$ resp. $\gamma_2$ such that $\alpha_1\cup\gamma_1$ and $\alpha_2\cup\gamma_2$ bound disjointed Jordan domains $D_1,D_2$ respectively.
	
	On one hand since the map $f$ acts as $z\mapsto z^k$ and the restriction $f|_{U}$ is one to one locally at $z$, we can assume that $f$ on the arcs $\gamma_1,\gamma_2,\alpha_1,\alpha_2$ is one to one, and that $f$ maps $D_1,D_2$ respectively into the same Fatou component $V:=f(U)$ with $f(D_1)\cap f(D_2)=\es$. Note that $V$ is a Jordan domain by the assumptions. On the other hand, the two domains $f(D_1),f(D_2)$ which are disjoint from Jordan curve $\partial V$ must be separated by $\partial V$, as their boundaries share a common Jordan arcs $f(\gamma_1)\subseteq \partial V$. Thus we get a contradiction. The proof is complete.
\end{proof}

\begin{lema}[{The boundaries of periodic Fatou components are quasicircles}]\label{periodic uniform quasicircles}
	Let $f$ be a rational map whose Fatou components are Jordan domains. Suppose that $U$ is a periodic Fatou component of $f$ whose boundary $\partial U$ contains neither parabolic periodic points nor the points in $\omega(c)$ for any recurrent critical point $c$. Then $\partial U$ is a quasicircle.
\end{lema}

\begin{proof}
	After iterating $f$ several times, we can assume that the periodic Fatou component $U$ is fixed by $f$. Without loss of generality, we suppose that $\infty\not\in \MJ_f$. Let $\epsilon=\delta_0>0$ be the number defined as in \eqref{delta0}. For any $x\in\partial U$, by Theorem \ref{mane lemma} and Lemma \ref{remk}, there exists an open neighborhood $U_x:=B(x,\delta_x)$ of $x$ satisfying (P1), (P2) and (P3). Since $\partial U$ is compact and $\partial U\subseteq\bigcup_{x\in\partial U}B(x,\delta_x/2)$, one can select a collection of finite number of elements $\mathcal{U}=\{B(x_1,\delta_{x_1}/2)$,$\cdots$,$B(x_N,\delta_{x_N}/2)\}$ such that $\partial U$ is covered by $\mathcal{U}$.
	Let $\delta_1>0$ be the Lebesgue number of $\mathcal{U}$. Then every subset of $\partial U$ with diameter $\leq\delta_1$ must be contained in at least one open disk $B(x_i,\delta_{x_i}/2)$ for some $1\leq i\leq N$.
	
	By Lemma \ref{local_homeomorphic}, the restriction of $f$ on $\partial U$ is a local homeomorphism. This means that there exists a number $\delta_2>0$ such that for any subset $E\subseteq\partial U$ with $\textup{diam}(E)\leq\delta_2$, the restriction of $f$ on $E$ is a homeomorphism. Recall that $\delta_{\partial U}>0$ is the number depending only on $\partial U$ which is defined in Lemma \ref{jordan_curve}. We define
	\begin{equation}\label{delta-choice}
	\delta:=\min\left\{\frac{\delta_1}{M},\delta_2,\frac{\delta_{\partial U}}{M}\right\},
	\end{equation}
	where $M:=1+\sup\{|f'(z)|:\textup{dist}(z,\MJ_f)\leq\delta_0\}<+\infty$. By (P2) in Theorem \ref{mane lemma}, there exists an integer $d_i>0$ for each $1\leq i\leq N$ such that for all $n\geq 0$ and every connected component $V$ of $f^{-n}(B(x_i,\delta_{x_i}))$, the degree of $f^{\circ n}:V\to B(x_i,\delta_{x_i})$ is $\leq d_i$. Define
	\begin{equation}\label{equ-max-d}
	d=\max_{1\leq i\leq N}d_i.
	\end{equation}
	
	Let $x,y$ be two different points in $\partial U$. We use $\gamma:=L(x,y)$ to denote one of the two components of $\partial U\setminus\{x,y\}$ with the smaller diameter. Now we divide the argument into two cases.
	
	\textbf{Case 1}: Suppose that $\textup{diam}(\gamma)\geq\delta$. Define $E:=\{(\xi,\eta)\in\partial U\times\partial U:\, \textup{diam}(L(\xi,\eta))\geq\delta\}$. Then $E$ is compact and $(\xi,\xi)\not\in E$. The function
	\begin{equation*}
	h: \partial U\times\partial U\to\R^+,  (\xi,\eta)\mapsto\frac{\textup{diam}(L(\xi,\eta))}{|\xi-\eta|}
	\end{equation*}
	is continuous on $E$. Then $h$ has a maximum $K_1$ on $E$ since $E$ is compact. In particular, the turning of $\gamma$ about $x$ and $y$ satisfies
	\begin{equation}\label{case_1}
	\Lambda(\gamma;x,y)=\frac{\textup{diam}(\gamma)}{|x-y|}\leq K_1.
	\end{equation}
	
	\textbf{Case 2}: Suppose that $\textup{diam}(\gamma)<\delta$.  Denote $\gamma_n:=f^{\circ n}(\gamma)$ for $n\geq 0$. Note that the forward orbit of $\gamma$ will eventually cover $\partial U$. There is a smallest integer $n\geq 0$ such that
	\begin{equation}\label{minimal_n}
	\textup{diam}(\gamma_n)<\delta\textup{~and~}\textup{diam}(\gamma_{n+1})=\diam(f(\gamma_n))\geq\delta.
	\end{equation}
	By the choice of $\delta$ in \eqref{delta-choice}, we know that $f^{\circ (n+1)}|_{\gamma}$ is a homeomorphism and so $\gamma_{n+1}$ is a Jordan arc connecting $f^{\circ(n+1)}(x)$ with $f^{\circ(n+1)}(y)$.
	Note that there exist two points $z_1,z_2\in\gamma_n$, such that
	\begin{equation}\label{estimate_length}
	\begin{split}
	\textup{diam}(\gamma_{n+1})
	&=  |f(z_1)-f(z_2)|\leq\int_{[z_1,z_2]}|f'(z)|\,|\textup{d}z|\\
	& \leq M|z_1-z_2|\leq M\,\textup{diam}(\gamma_n)\leq M\delta\leq\textup{min}\{\delta_1,\delta_{\partial U}\},
	\end{split}
	\end{equation}
	where $[z_1,z_2]$ is the straight segment connecting $z_1$ and $z_2$.
	
	By the definition of $\delta_{\partial U}$ and Lemma \ref{jordan_curve}, the Jordan arc $\gamma_{n+1}$ is one of the two components of $\partial U\setminus \{ f^{\circ (n+1)}(x), f^{\circ (n+1)}(y)\}$ with smaller diameter. Since $\diam(\gamma_{n+1})\geq\delta$ by \eqref{minimal_n}, as discussed in Case 1 above, we have
	\begin{equation}\label{gamma}
	\Lambda(\gamma_{n+1};f^{\circ(n+1)}(x),f^{\circ(n+1)}(y))\leq K_1.
	\end{equation}
	
	By the definition of $\delta_1$, there is a disk $B(x_i,\delta_{x_i}/2)$ such that $\gamma_{n+1}\subseteq B(x_i,\delta_{x_i}/2)$ for some $1\leq i\leq N$ since $\diam(\gamma_{n+1})\leq\delta_1$ by \eqref{estimate_length}. Let $B_{n+1}(x_i,\delta_{x_i}/2)$ and $B_{n+1}(x_i,\delta_{x_i})$, respectively, be the components of the preimages $f^{-(n+1)}(B(x_i$, $\delta_{x_i}/2))$ and $f^{-(n+1)}(B(x_i,\delta_{x_i}))$ both containing $\gamma$. Note that both of them are simply connected by the choice of $\delta_{x_i}$.
	Applying Lemma \ref{control turning} to the case $(U_1,V_1)$ $=$ $(B_{n+1}(x_i,\delta_{x_i}/2)$, $B_{n+1}(x_i,\delta_{x_i}))$, $(U_2,\,V_2)=(B(x_i,\delta_{x_i}/2),\,B(x_i,\delta_{x_i}))$ and $m=\frac{1}{2\pi}\textup{log}\,2$, together with \eqref{equ-max-d} and \eqref{gamma}, we have
	\begin{equation}\label{case_2}
	\Lambda(\gamma;x,y)\leq C_2(d,m)\Lambda(\gamma_{n+1};f^{\circ(n+1)}(x),f^{\circ(n+1)}(y))\leq C_2(d)K_1,
	\end{equation}
	where $C_2(d)$ is a constant depending only on $d$. Let
	\begin{equation*}
	K=K_1(1+C_2(d)).
	\end{equation*}
	Then $\Lambda(\gamma;x,y)\leq K$ holds for any different $x,y\in\partial U$ by \eqref{case_1} and \eqref{case_2}. The arbitrariness of $x$ and $y$ implies that $\partial U$ is a quasicircle.
\end{proof}

Now we need to determine when the boundaries of all the Fatou components are uniform quasicircles. According to Sullivan \cite{Sul85}, each Fatou component of a rational map is eventually periodic. It is natural to consider the pull back of the periodic Fatou components and then use some distortion lemmas to control the shape of pre-periodic Fatou components. To do this, it is necessary to construct a larger simply connected domain surrounding the periodic Fatou component such that all components of its preimages under the $n$-th iteration are still simply connected.

\begin{prop}[{Uniform quasicircles I}]\label{uniform_quasicircle-1}
	Let $f$ be a rational map whose Fatou components are Jordan domains. Suppose that all the boundaries of periodic Fatou components are disjoint from the $\omega$-limit sets of the critical points. Then the boundaries of all Fatou components of $f$ are uniform quasicircles.
\end{prop}

\begin{proof}
	If all periodic Fatou components of $f$ are disjoint from the $\omega$-limit sets of the critical points, then $f$ has no parabolic periodic points (see \cite[Theorem 10.15]{Mil06}). By Lemma \ref{periodic uniform quasicircles} and Sullivan's eventually periodic theorem, all the boundaries of the Fatou components of $f$ are quasicircles. We only need to prove that they are uniform quasicircles.
	
	Let $\mathcal{U}'$ be the collection of all the Fatou components such that each of them is either a critical Fatou component (contains at least one critical point) or a periodic Fatou component. We use $\mathcal{U}:=\MO^+(\mathcal{U}')=\{U_1,\cdots,U_n\}$ to denote the union of the forward orbits of all the Fatou components in $\mathcal{U}'$. Note that the number of Fatou components in $\mathcal{U}$ is finite since $\mathcal{U}'$ is. Therefore, there exists a constant $K'>1$ such that $\partial U_i$ is a $K'$-quasicircle for all $1\leq i\leq n$. For each $1\leq i\leq n$, let $V_i$ be a Jordan disk such that $V_i\setminus\overline{U}_i$ is an annulus which is disjoint from the forward orbits of the critical points.
	
	Let $m_i=\textup{mod}(V_i\setminus\overline{U}_i)>0$ for $1\leq i\leq n$. For each Fatou component $U\not\in\mathcal{U}$, there exists a minimal number $k\geq 1$ such that $f^{\circ k}(U)=U_i\in\mathcal{U}$ for some $i$. Let $V$ be the component of $f^{-k}(V_i)$ containing $U$. Then $f^{\circ k}:V\to V_i$ is a conformal map and $V$ is a Jordan disk since $V_i$ contains no points in the critical orbits. By Corollary \ref{pullback of quasicircle}, the boundary $\partial U$ is a $C(K',m_i)$-quasicircle, where $C(K',m_i)$ is a constant depending only on $K'$ and $m_i$.
	
	Let $K=\max_{1\leq i\leq n}C(K',m_i)$. Then the boundary of each Fatou component of $f$ is a $K$-quasicircle. By the arbitrariness of $U$, this means that the boundaries of the Fatou components of $f$ are uniform quasicircles.
\end{proof}

By the assumption in Proposition \ref{uniform_quasicircle-1}, the periodic Fatou components can only be attracting or super-attracting. Recall that a rational map $f$ is called \emph{semi-hyperbolic} if the Julia set $\MJ_f$ contains neither parabolic periodic points nor recurrent critical points.

\begin{prop}[{Uniform quasicircles II}]\label{uniform_quasicircle}
	Let $f$ be a \textup{semi-hyperbolic} rational map such that the boundary of each Fatou component is a Jordan curve. Then the boundaries of all the Fatou components of $f$ are uniform quasicircles.
\end{prop}

It seems that the condition in Proposition \ref{uniform_quasicircle-1} is stronger than that in Proposition \ref{uniform_quasicircle}. However, this is not the case. Actually, one can construct a rational map with a recurrent critical point, whose $\omega$-limit set is disjoint from the boundaries of any Fatou components, using similar method as stated in \S\ref{example}.

\begin{proof}[{Proof of Proposition \ref{uniform_quasicircle}}]
	By Lemma \ref{periodic uniform quasicircles} and Sullivan's eventually periodic theorem, it follows that all the boundaries of the Fatou components of $f$ are quasicircles since $f$ is semi-hyperbolic. We only need to prove that they are uniform quasicircles. According to \cite[Theorem 1.2]{Yin99}, the Julia set $\MJ_f$ is locally connected. Then for any $\epsilon>0$, there are only finitely many Fatou components with diameter $\geq\epsilon$ \cite[Lemma 19.5]{Mil06}.
	
	Without loss of generality, we suppose that $\infty\not\in \MJ_f$. Let $\epsilon=\delta_0>0$ be the number defined as in \eqref{delta0}. For any $x\in \MJ_f$, by Theorem \ref{mane lemma} and Lemma \ref{remk}, there exists an open neighborhood $U_x:=B(x,\delta_x)$ of $x$ satisfying (P1), (P2) and (P3). Since $\MJ_f$ is compact, there exists a collection of finite number of elements $\mathcal{U}=\{B(x_1,\delta_{x_1}/2)$,$\cdots$,$B(x_N,\delta_{x_N}/2)\}$ such that $\MJ_f$ is covered by $\mathcal{U}$. We use $\delta>0$ to denote the Lebesgue number of $\mathcal{U}$. Then every subset of $\MJ_f$ with diameter $\leq\delta$ must be contained in at least one open disk $B(x_i,\delta_{x_i}/2)$ for some $1\leq i\leq N$.
	
	We divide the collection of all the Fatou components $\mathcal{F}$ of $f$ into two classes as following. Let $\mathcal{F}_0$ be the collection of all the Fatou components such that each $U\in\mathcal{F}_0$ is one of the following cases: either (i) $U$ contains at least one critical point; or (ii) $U$ is periodic; or (iii) $\diam(U)\geq \delta$. Let $\mathcal{F}_1':=\MO^+ (\mathcal{F}_0)$ be the set of the union of the forward orbits of all the Fatou components in $\mathcal{F}_0$. Define $\mathcal{F}_1:=\mathcal{F}_1'\cup f^{-1}(\mathcal{F}_1')$. By Sullivan's eventually periodic theorem, the number of Fatou components in $\mathcal{F}_1$ is finite since $\mathcal{F}_0$ is also. Therefore, there exists a constant $K'>1$ such that each Fatou component in $\mathcal{F}_1$ is a $K'$-quasicircle.
	
	For any Fatou component $U\in\mathcal{F}\setminus\mathcal{F}_1$, we have $\diam(U)<\delta$. There exists a minimal integer $n_U\geq 1$ such that $f^{\circ n_U}(U)\in f^{-1}(\mathcal{F}_1')\setminus \mathcal{F}_1' \subseteq \mathcal{F}_1$ and $\diam (f^{\circ n_U}(U))<\delta$. Moreover, the map $f^{\circ n_U}: U\to f^{\circ n_U}(U)$ is conformal. By the definition of $\delta$, there exists some disk $B(x_i,\delta_{x_i}/2)$ in $\mathcal{U}$ such that $f^{\circ n_U}(U)\subseteq B(x_i,\delta_{x_i}/2)$.
	We use $B_U$ and $B_U'$, respectively, to denote the components of $f^{-n_U}(B(x_i,\delta_{x_i}/2))$ and $f^{-n_U}(B(x_i,\delta_{x_i}))$ both containing $U$.
	
	Let $x,y\in \partial U$ be two different points and $\gamma_1$, $\gamma_2$ the two different components of $\partial U\setminus\{x,y\}$. Then $f^{\circ n_U}(\gamma_1)$ and $f^{\circ n_U}(\gamma_2)$ are both Jordan arcs connecting $f^{\circ n_U}(x)$ with $f^{\circ n_U}(y)$. Applying Lemma \ref{control turning}(2) to the case $(U_1,V_1)=(B_U,B_U')$, $(U_2,V_2)=(B(x_i,\delta_{x_i}/2)$, $B(x_i,\delta_{x_i}))$, $m=\frac{1}{2\pi}\log 2$, $g=f^{\circ n_U}$ and $E=\gamma_j$, where $j=1,2$, we have
	\begin{equation*}
	\Lambda(\gamma_j;x,y)\leq C_2(d_i)\,\Lambda(f^{\circ n_U}(\gamma_j);f^{\circ n_U}(x),f^{\circ n_U}(y)),
	\end{equation*}
	where $C_2(d_i)$ is a constant depending only on $d_i$ and $d_i>0$ is the number appeared in Theorem \ref{mane lemma} which depends on $x_i$. Then
	\begin{equation*}
	\min_{j\in\{1,2\}}\{\Lambda(\gamma_j;x,y)\}\leq C_2(d_i)\,\min_{j\in\{1,2\}}\{\Lambda(f^{\circ n_U}(\gamma_j);f^{\circ n_U}(x),f^{\circ n_U}(y))\} \leq C_2(d_i)K'.
	\end{equation*}
	Let $K=\max_{1\leq i\leq N}C_2(d_i)K'$. Then $\partial U$ is a $K$-quasicircle by the arbitrariness of $x$ and $y$. By the arbitrariness of $U$, we know that each Fatou component of $f$ is a $K$-quasicircle and $K$ is a constant depending only on $f$.
\end{proof}

Note that in Propositions \ref{uniform_quasicircle-1} and \ref{uniform_quasicircle}, we do not require that the closures of the Fatou components are mutually disjoint. This means that these two propositions can be applied also to those rational maps whose Julia sets are not Sierpi\'{n}ski carpets.

\subsection{Sufficiency for the property of uniformly relatively separated}

By Lemma \ref{modulus and separated}, if the lower bound of the annuli between the boundaries of the Fatou components can be controlled, then one can prove that the peripheral circles of the carpet Julia set are uniformly relatively separated.

\begin{prop}[{Uniformly relatively separated}]\label{separated}
	Let $f$ be a rational map whose Julia set $\MJ_f$ is a Sierpi\'{n}ski carpet. If all the boundaries of periodic Fatou components are disjoint from the $\omega$-limit sets of the critical points in $\MJ_f$, then the boundaries of Fatou components are uniformly relatively separated.
\end{prop}

\begin{proof}
	Let $\mathcal{U}=\{X_1,\cdots,X_n\}$ be the collection of the all the periodic Fatou components of $f$. After iterating $f$ by some times if necessary, we can assume that each $X_i$ have period precise one. For $1\leq i\leq n$, let $Y_i$ be a simply connected domain containing $X_i$ such that $Y_1,\cdots,Y_n$ are mutually disjoint and each annulus $A_i:=Y_i\setminus \overline{X}_i$ contains no points in the critical orbits. Denote $m:=\min_{1\leq i\leq n}\textup{mod}(A_i)>0$.
	
	For any two different Fatou components $U_1$ and $U_2$, there exist two minimal numbers $n_1,n_2\geq 0$ such that $f^{\circ n_1}(U_1),f^{\circ n_2}(U_2)\in\mathcal{U}$. In particular, there exist two integers $1\leq k_1,k_2\leq n$ such that $f^{\circ n_1}(U_1)=X_{k_1}$ and $f^{\circ n_2}(U_2)=X_{k_2}$. Since the annulus $A_{k_i}$ contains no critical values of $f^{\circ n_i}$ for $i\in\{1,2\}$, the restriction of $f^{\circ n_i}$ on each component of $f^{-n_i}(A_{k_i})$ is an unbranched covering. By Riemann-Hurwitz's formula, it follows that each component of their preimages is still an annulus. Therefore, there exist two simply connected domains $V_1$ and $V_2$ surrounding $U_1$ and $U_2$ respectively, such that $V_i\setminus\overline{U}_i$ is a component of $f^{-n_i}(A_{n_i})$ and $\textup{deg}(f^{\circ n_i}:{V_i\to Y_{n_i}})=\textup{deg}(f^{\circ n_i}:U_i\to X_{n_i})$. Note that $f^{\circ j_1}(U_i)\cap f^{\circ j_2}(U_i)=\emptyset$ for $0\leq j_1<j_2\leq n_i$ (if $n_i\geq 1$) and $f$ has only finitely many critical points. So the degree of $f^{\circ n_i}|_{U_i}$ is bounded by some number $N\geq 1$ depending only on $f$. Denote by $A$ the annulus bounded by $\partial U_1$ and $\partial U_2$ in $\EC$.
	We now divide the arguments into two cases.
	
	\textbf{Case 1}: Suppose that $n_1=n_2$. Then $V_1$ and $V_2$ are two disjoint components of $f^{-n_1}(Y_{k_1}\cup Y_{k_2})$. By Lemma \ref{lem:KL09}, we have
	\begin{equation*}
	\textup{mod}(A)\geq\textup{mod}(V_1\setminus\overline{U}_1)+\textup{mod}(V_2\setminus\overline{U}_2)\geq
	\textup{mod}(A_{k_1})/N+\textup{mod}(A_{k_2})/N\geq 2m/N.
	\end{equation*}
	
	\textbf{Case 2}: Suppose that $n_1>n_2$. We claim that $V_1$ and $U_2$ are disjoint. Otherwise, the annulus $V_{1}\setminus\overline{U}_1$ intersects $U_2$ and so $f^{\circ n_2}(V_{1}\setminus\overline{U}_1)$ intersects the fixed Fatou component $X_{k_2}$. Then $A_{k_1}=f^{\circ (n_1-n_2)}(f^{\circ n_2}(V_1\setminus\overline{U}_1))$ joints with $X_{k_2}$, which contradicts with the choice of $A_{k_1}$. Thus we have
	\begin{equation*}
	\textup{mod}(A)\geq \textup{mod}(V_1\setminus\overline{U}_1)\geq {m/N}.
	\end{equation*}
	
	Above all, the annulus $A$ has modulus not less than $m/N$. By Lemma \ref{modulus and separated}, $U_1$ and $U_2$ are relatively separated with the relative distance $\Delta(\partial U_1,\partial U_2)$ depending only on $m$ and $N$. By the arbitrariness of $U_1$ and $U_2$, the peripheral circles of the carpet Julia set are uniformly relatively separated.
\end{proof}

Note that the condition in Proposition \ref{separated} does not exclude the existence of parabolic points on the Julia set.

\vskip0.2cm
\noindent{\textit{Proof of Theorem \ref{thm-main-1}}}. By Propositions \ref{uniform_quasicircle-1} and \ref{separated}, the peripheral circles of carpet $\MJ_f$ are uniform quasicircles and uniformly relatively separated. According to Bonk \cite[Corollary 1.2]{Bon11}, $\MJ_f$ is quasisymmetrically equivalent to a round carpet.
\hfill $\square$

\subsection{The property of non-uniformly relatively separated}

If the peripheral circles of a carpet Julia set are uniformly relatively separated, a natural question is whether it implies that all the boundaries of pre-periodic Fatou components avoid the accumulation points of the critical orbits in the Julia set. We give the answer in the following proposition.

\begin{prop}[{Non-uniformly relatively separated}]\label{condi-necessary}
	Let $f$ be a semi-hyperbolic rational map whose Julia set is a Sierpi\'nski carpet. Suppose that there exists a Fatou component $U$ of $f$ such that $\partial U\cap\omega(c)\neq\emptyset$ for some critical point $c\in \MJ_f$. Then the boundaries of Fatou components of $f$ are not uniformly relatively separated.
\end{prop}

\begin{proof}
	Without loss of generality, we suppose that $\infty\not\in \MJ_f$. Let $\epsilon=\delta_0>0$ be the number defined as in \eqref{delta0}. Let $x\in\partial U\cap\omega(c)$. By Theorem \ref{mane lemma} and Lemma \ref{remk}, there exists a number $\delta_x>0$ such that the open neighborhood $B(x,\delta_x)$ satisfies (P1), (P2) and (P3). Suppose $c_{k_n}:=f^{\circ k_n}(c)$ tends to $x$. Set $\epsilon_{k_n}:=|x-c_{k_n}|$. Clearly $\epsilon_{k_n}\to 0$ as $n\to\infty$.
	
	For a given $0<\delta<\delta_x$, there exists an integer $N\geq 1$ such that for any $n\geq N$, $B(c_{k_n},\epsilon_{k_n})\subseteq B(c_{k_n},\delta)\subseteq B(x,\delta_x)$. Let $\gamma_0$ be the component of $\partial U\cap B(c_{k_n},\delta)$ containing $x$ with endpoints $a_0$ and $b_0$. Denote $X_0:=B(c_{k_n},\epsilon_{k_n})$ and $Y_0:=B(c_{k_n},\delta)$ for some $n\geq N$. After Pulling $X_0$, $Y_0$ back by $f^{\circ (k_n-1)}$ resp. $f^{\circ k_n}$, we denote by $X_{k_n-1},Y_{k_n-1}$ resp. $X_{k_n}$, $Y_{k_n}$ the simply connected components of their preimages containing the critical value $c_1$ resp. the critical point $c$. See Figure \ref{fig-sketch}.
	
\begin{figure}[!htpb]
  \setlength{\unitlength}{1mm}
  \centering
  \includegraphics[width=125mm]{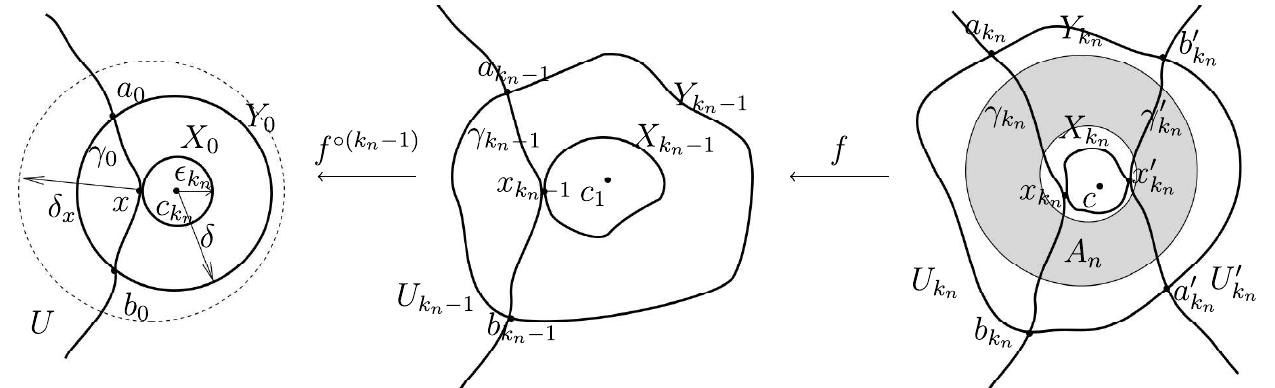}
  \caption{Sketch illustration of the idea in the proof of Proposition \ref{condi-necessary}. One main point in the proof is to verify the existence of two different Fatou components $U_{k_n}$ and $U_{k_n}'$ such that they are both $k_n$-th preimages of the given Fatou component $U$ and the distance between $U_{k_n}$ and $U_{k_n}'$ tends to zero as $n$ tends to infinity.}
  \label{fig-sketch}
\end{figure}		
	
	Consider the branched covering $f^{\circ (k_n-1)}:(X_{k_n-1},Y_{k_n-1})\to (X_0,Y_0)$. Since the boundary of every Fatou component cannot contain any critical point if the Julia set is a Sierpi\'{n}ski carpet, it means that $\gamma_{0}$ avoids the postcritical sets. There exist $x_{k_n-1}$, $\gamma_{k_n-1}$ and $U_{k_n-1}$ which are components of the preimages of $x, \gamma_0$ and $U$ respectively such that $x_{k_n-1}\in\gamma_{k_n-1}\cap\partial X_{k_n-1}$ and $\gamma_{k_n-1}$ is a subarc of the Fatou component $U_{k_n-1}$ with endpoints $a_{k_n-1}, b_{k_n-1}$. Applying the Maximum Value Principle on  $f^{\circ (k_n-1)}|_{Y_{k_n-1}}$, we have $a_{k_n-1},b_{k_n-1}\in \partial Y_{k_n-1}$.
	
	Now we consider the other branched covering $f:(X_{k_n},Y_{k_n})\to (X_{k_n-1},Y_{k_n-1})$. Since $\tu{deg}(f|_{X_{k_n}})\geq 2$, there exist at least two different points $x_{k_n}$, $x'_{k_n}\in f^{-1}(x_{k_n-1})$ lying on $\partial X_{k_n}$, and also two disjoint Jordan arcs $\gamma_{k_n}\subseteq\partial U_{k_n},\gamma_{k_n}'\subseteq\partial U'_{k_n}$ containing $x_{k_n},x'_{k_n}$ respectively, where $U_{k_n}$ and $U_{k_n}'$ are Fatou components of $f$. The subarcs $\gamma_{k_n}$ and $\gamma_{k_n}'$ meet $\partial Y_{k_n}$ at their endpoints $\{a_{k_n},b_{k_n}\}$ and $\{a'_{k_n},b'_{k_n}\}$.
	
	\textbf{Claim}: For sufficiently large $n\geq N$, the Fatou components $U_{k_n}, U_{k_n}'$ are distinct.
	
	Indeed, as $\epsilon_n\to 0$, the modulus $\tu{mod}(Y_0\setminus\ol{X}_0)\to\infty$ and thus $\tu{mod}(Y_{k_n}\setminus\ol{X}_{k_n})\to\infty$, according to Lemma \ref{lem:KL09} and Theorem \ref{mane lemma}. Then for sufficiently large $n$ there exists a round annulus $A_n:=\{z\in\mb{C}:\,0<r_n<|z|<r_n'\}$ which is \emph{essentially} contained in $Y_{k_n}\sm \ol{X}_{k_n}$ such that
	$$\tu{mod}(A_n)=\frac{1}{2\pi}\tu{log}\frac{r'_n}{r_n}\to+\infty$$
	as $n\to \infty$ (see \cite[Theorem 2.1, p.\,10]{McM94}).
	
	We argue by contradiction and assume $U_{k_n}=U_{k_n}'$. Then for any component $\gamma$ of $\partial U_{k_n}\setminus\{x_{k_n},x_{k_n}'\}$, the turning of $\gamma$ satisfies
	\begin{equation}\label{distant}
	\begin{split}
	\Lambda(\gamma;x_{k_n},x_{k_n}')
	\geq &~\frac{\tu{min}\{\tu{diam}\gamma_{k_n,1},\tu{diam}\gamma_{k_n,2},\tu{diam}\gamma'_{k_n,1},\tu{diam}\gamma'_{k_n,2}\}}{|x_{k_n}-x_{k_n}'|}\\
	\geq &~\frac{r'_n-r_n}{2r_n}\to \infty \text{ as } n\to\infty,
	\end{split}
	\end{equation}
	where the subarcs $\gamma_{k_n,1},\gamma_{k_n,2}$ and $\gamma'_{k_n,1},\gamma'_{k_n,2}$ are the components of $\gamma_{k_n}\setminus\{x_{k_n}\}$ and $\gamma'_{k_n}\setminus\{x_{k_n}'\}$ respectively. The second inequality holds because each of the four subarcs connects the two complement components of $A_n$.
	By Proposition \ref{uniform_quasicircle}, the family of peripheral circles of the Fatou set $\mathcal{F}_f$ are uniformly quasicircles. Thus we get a contradiction by \eqref{distant}. The claim follows.
	
	We now estimate the relative distance between $\partial U_{k_n}$ and $\partial U_{k_n}'$. Note that the bounded component of $\mathbb{C}\setminus A_n$ covers $X_{k_n}$. Therefore, we have
	\begin{equation*}
	\begin{split}
	\Delta(\partial U_{k_n},\partial U_{k_n}')\leq ~ & \frac{|x_{k_n}-x_{k_n}'|}{\min\{\diam(\partial U_{k_n}), \diam(\partial U_{k_n}')\}} \\
	\leq ~ & \frac{|x_{k_n}-x_{k_n}'|}{\tu{min}\{\tu{diam}\gamma_{k_n,1},\tu{diam}\gamma_{k_n,2},\tu{diam}\gamma'_{k_n,1},\tu{diam}\gamma'_{k_n,2}\}}
	\\
	\leq ~ & \frac{2r_n}{r_n'-r_n}\to 0 \text{ as } n\to\infty.
	\end{split}
	\end{equation*}
	This means that the boundaries of Fatou components of $f$ are not uniformly relatively separated. The proof is complete.
\end{proof}

\vskip0.2cm
\noindent{\textit{Proof of Theorem \ref{thm-main-2}}}. This follows immediately by Propositions \ref{uniform_quasicircle} -- \ref{condi-necessary}.
\hfill $\square$

\subsection{The locally porous property of carpet Julia sets}

Let $f$ be a semi-hyperbolic rational map whose Julia set $\mathcal{J}_f$ is a carpet. We use $\mathcal{C}_f$ to denote the family of corresponding peripheral circles. By Theorem \ref{uniform_quasicircle}, the family $\mathcal{C}_f$ consists of uniform quasicircles. In order to describe all the locations and scales of curves in $\mathcal{C}_f$, we need the definition of locally porous by following \cite[p.\,4323]{Mer14} (see also \cite[p.\,32]{BLM16}).

The carpet Julia set $\mathcal{J}_f$ is said to be \emph{locally porous}\footnote{It seems that the definition here is slightly different from that in \cite[p.\,4323]{Mer14}. However, they are obviously equivalent.} provided that there exists $0<c<1$ such that for any $z\in \mathcal{J}_f$ and any $0<r\leq 2$, there exists a circle $\gamma$ in $\mathcal{C}_f$ such that $\gamma\subseteq B_{\sigma}(z, r)$
and $cr\leq \text{diam}_{\sigma}(\gamma)$. Since the elements in family $\MC_f$ are uniform quasicircles and they have uniformly bounded shape, it is not difficult to check that the definition of locally porous here implies that $\mathcal{J}_f$ is \emph{porous}  (i.e. \emph{shallow} in the sense of McMullen, see also \cite{Yin00}) and the peripheral circles \emph{occur on all locations and scales} \cite[p.\,24]{BLM16}.

\begin{thm}\label{local_porous}
	Let $f$ be a semi-hyperbolic rational map whose Julia set $\mathcal{J}_f$ is a carpet. Then $\mathcal{J}_f$ is locally porous.
\end{thm}

\begin{proof}
	Since M\"{o}bius transformations are uniformly continuous homeomorphisms with respect to spherical metric. For simplicity, we can assume $\mathcal{J}_f$ is contained in $\mathbb{D}$ by conjugating a suitable M\"{o}bius transformation and prove the locally porous property in the
	sense of Euclidean metric.
	
	Firstly, for $0<r<1$ we define two functions
	$$h_1(r):=\mathop{\text{inf}}^{}_{z\in\mathcal{J}_f}\text{inf}\{s>0; \text{diam}f(B(z,s))\geq r\}/r$$
	and
	$$h_2(r):=\mathop{\text{inf}}^{}_{z\in\mathcal{J}_f}\mathop{\text{sup}}^{}_{\gamma\in \mathcal{C}_f}\{\text{\diam}(\gamma);\gamma\subseteq B(z,r)\}/r.
	$$
	One can easily check that both of the functions are positive for any $r\in (0,1)$.
	
	Since $f$ is a semi-hyperbolic rational map, according to \cite{Yin99}, there exist a positive constant $\delta_0>0$ and an integer $d\geq 1$ such that for every $z\in\MJ_f$, $n\geq 0$ and every connected component $V$ of $f^{-n}(B(z,2\delta_0))$, the degree of $f^{\circ n}:V\to B(z,2\delta_0)$ is $\leq d$. Moreover, we assume that $\delta_0$ is chosen small enough such that each $V$ is simply connected by (P3) in Lemma \ref{remk}.
	
	Given $z\in \mathcal{J}_f$ and $0<r<\delta_0/3$, we denote $z_n:=f^{\circ n}(z)$, $L_n:=\text{max}\{|z_n-w|;w\in\partial f^{\circ n}(B(z,r))\}$ and
	$l_n:=\text{min}\{|z_n-w|;w\in\partial f^{\circ n}(B(z,r))\}$. By the expansion of $f$ on the Julia set, there exists a minimal integer $n_0\geq 0$ such that
	$$\text{diam}f^{\circ (n_0+1)}(B(z,r))\geq \delta_0.$$
	Then $f^{\circ n_0}(B(z,r))\subseteq B(z_{n_0},L_{n_0})\subseteq B(z_{n_0},\delta_0)$ and $\text{diam}f(B(z_{n_0},L_{n_0}))\geq \delta_0$. Let $U$ be a component of $f^{-{n_0}}(B(z_{n_0},2\delta_0))$ containing $B(z,r)$.
	By the choices of $h_1$ and $\delta_0$ and by applying the Lemma \ref{forward_dis} to the proper holomorphic map $f^{n_0}:U\to B(z_{n_0},2\delta_0)$, there exists a constant $K_1$ depending only on $d$ such that
	\begin{equation}\label{rato1}
	K_2:=h_1(\delta_0)\leq \frac{L_{n_0}}{\delta_0}\leq\frac{K_1l_{n_0}}{\delta_0}\leq K_1.
	\end{equation}
	Let $\gamma_{n_0}$ be a peripheral circle in the disk $B(z_{n_0}, K_2\delta_0/K_1)$ such that
	\begin{equation}\label{rato2}
	K_3:=h_2(K_2\delta_0/K_1)\leq \frac{K_1\text{diam} \gamma_{n_0}}{K_2\delta_0}=\frac{K_1|x_{n_0}-y_{n_0}|}{K_2\delta_0}\text{ for some }
	x_{n_0},y_{y_0}\in \gamma_{n_0}.
	\end{equation}
	
	Now we consider the pullback of $B(z_{n_0},l_{n_0})$ by $f^{\circ n_0}$. Let $B$ and $\gamma$, respectively, be two components of
	$f^{-n_0}(B(z_{n_0},l_{n_0}))$ and $f^{-n_0}(\gamma_{n_0})$ in $B(z,r)$. Clearly, we have $r\leq \text{diam} B\leq 2r$.
	Let $x,y\in \gamma$ with
	$f^{\circ n_0}(x)=x_{n_0}$ and $f^{\circ n_0}(y)=y_{n_0}$. According to Lemma \ref{control turning}(2), \eqref{rato1} and \eqref{rato2}, there exists a constant $K_4>0$ depending only on $d$ such that
	$$
	\frac{r}{\text{diam}\gamma}\leq \frac{\text{diam} B}{|x-y|}\leq K_4\frac{\text{diam}f^{\circ n_0}(B)}{|x_{n_0}-y_{n_0}|}=\frac{2K_4l_{n_0}}{\text{diam}\gamma_{n_0}}\leq \frac{2K_1K_4}{K_2K_3}=:\frac{1}{c}.
	$$
	
	By the arbitrariness of $z$ and $r$, the proof is finished.
\end{proof}

\section{Dynamics on the Fatou components}\label{sec-dyn-Fatou}

In this section, we describe the dynamics of a rational map on its Fatou components when the corresponding Julia set is locally connected. Recall that $\text{Comp}(\mathcal{F}_f)$ is the collection of all the Fatou components of a given rational map $f$.

The following lemma extends the results in \cite[Lemma 3.3]{BLM16} to a broader range of situations.

\begin{lema}[Dynamics on the Fatou components]\label{dynamic_Fatou}
	Let $f$ be a rational map whose Julia set $\mathcal{J}_f$ is locally connected. Suppose that the post-critical set $\MP_f $ is finite on $\mathcal{F}_f$, i.e., $\sharp(\MP_f\cap\MF_f)<\infty$.
	Then there exists $\{(U, \phi_U)\}_{U\in\text{Comp}(\mathcal{F}_f)}$ such that
	
	$(1)$ Each $\phi_U:{U}\to{\mathbb{D}}$ is conformal.
	
	$(2)$ Assume that $U$ is a $p$-periodic Siegel disk and $\lambda:=e^{2\pi i\theta}$ is the rotation number of $f^{\circ p}|_U$ with $\theta$ irrational. Then $\phi_{f(U)}\circ f\circ\phi_U^{-1}=P_{p,\theta}$, where $P_{p,\theta}:z\mapsto e^{2\pi i\theta/p}z$.
	
	$(3)$ For any other Fatou component $U$ which is not a periodic Siegel disk, we have $\phi_{f(U)}\circ f\circ\phi_{U}^{-1}=P_{\delta_U}$, where $\delta_U:=\deg(f|_U)$ and $P_{\delta_U}: z\mapsto z^{\delta_U}$.
\end{lema}

\begin{proof}
	Firstly, since the Julia set is locally connected, then $\mathcal{J}_f$ is connected \cite[Corollary 4.15]{Mil06}. Thus $\mathcal{F}_f$ has no Herman rings and every Fatou component is simply connected. Since $\sharp(\MP_f\cap\MF_f)<\infty$, it follows that the periodic Fatou components of $f$ can only be super-attracting basins or Siegel disks (see \cite[\S\S8, 10]{Mil06}).
	
	Let $(U_0, z_0)$, $\cdots$, $(U_{p-1}, z_{p-1})$, $(U_p,z_p):=(U_{0},z_{0})$ be a periodic cycle of super-attracting basins with $z_k\in U_k$ such that  $f^{\circ k}(U_0, z_0)=(U_k,z_k)$. We denote $\delta_k:=\deg(f|_{U_k})$ and $h_k:=P_{\delta_k}$.
	According to \cite[Theorem 9.3]{Mil06}, there exists a global B\"{o}ttcher coordinate $\phi_{U_0}:U_0\to\D$ of $f^{\circ p}$ on $U_0$ such that
	\begin{equation}\label{cordinate}
	\phi_{U_0}\circ f^{\circ p}\circ \phi_{U_0}^{-1}=h:=h_{p-1}\circ \cdots\circ h_0, \text{ and } \phi_{U_0}(z_0)=0.
	\end{equation}
	Choose $\xi_0\in U_0\setminus\{z_0\}$ and $\eta_0:=\phi_{U_0}(\xi_0)\in\D\setminus\{0\}$. Set $\xi_k:=f^{\circ k}(\xi_0)$ and $\eta_k:=h_{k-1}\circ\cdots\circ h_{0}(\eta_0)$ for $1\leq k\leq p$.
	
	In what follows we try to construct $\phi_{U_k}$ by induction. For $k=1,\cdots, p$,
	by considering the unbranched coverings
	$$f:U_{k-1}\setminus\{z_{k-1}\} \to U_{k}\setminus\{z_{k}\}\text{ and }
	h_{k-1}:\mathbb{D}\setminus\{0\}\to\mathbb{D}\setminus\{0\},$$
	there exists a unique lift $\phi_{U_k}: U_k\to\mathbb{D}\setminus\{0\}$ of $\phi_{U_{k-1}}$ such that $h_{k-1}\circ\phi_{U_{k-1}}=\phi_{U_k}\circ f$ and $\phi_{U_k}(\xi_k)=\eta_k$ \cite[Propositions 1.33 and 1.34]{Hat02}. Comparing the degrees, we know that $\phi_{U_k}$ is conformal. Moreover, by the definition of the lifts, we have $\phi_{U_k}(z_k)=0$. The uniqueness of lift also guarantees that $\phi_{U_p}=\phi_{U_0}$.
	
	Let $(U_0, z_0)$, $\cdots$, $(U_{p-1}, z_{p-1})$, $(U_p,z_p):=(U_{0},z_{0})$ be a periodic cycle of Siegel disks with $z_k\in U_k$ such that $f^{\circ k}(U_0, z_0)=(U_k,z_k)$. Let $e^{2\pi i\theta}:=(f^{\circ p})'(z_0)$ be the multiplier and $h_k:=P_{p, \theta}:z\mapsto e^{2\pi i\theta/p}z$. By \cite[Lemma 11.1]{Mil06}, there also exists a global linearization $\phi_{U_0}$ of $f^{\circ p}$ on $U_0$ satisfying Equation \eqref{cordinate}. Then by exactly the same argument as in super-attracting case, we can obtain $\phi_{U_k}$, $1\leq k\leq p-1$, as desired.
	
	Now we can construct $\phi_U$ for any pre-periodic Fatou component $U$. To see this, by induction for $k=1,2,\cdots n$ and so on, suppose that $U$ is $k$-preperiodic, i.e., $k$ is the minimal positive integer such that $U_k:=f^{\circ k}(U)$ is periodic. Without loss of generality, we assume that $k=1$. As before, lifting $\phi_{U_1}:U_1\to \mathbb{D}$ via the coverings $f:U\to U_1$ and $P_{\delta_U}:\mathbb{D}\to \mathbb{D}$, we can get a conformal mapping $\phi_U:U\to \mathbb{D}$ such that $\phi_{U_1}\circ f=P_{\delta_U}\circ \phi_{U}$.
\end{proof}

\begin{rmk}
	If $\mathcal{J}_f$ is a carpet, then each Fatou component of $f$ is a Jordan domain. By Carath\'{e}odory's theory, one can extend $\phi_U:U\to\D$ to a homeomorphism from $\overline{U}$ to $\overline{\mathbb{D}}$. Clearly the extension $\phi_U$ still satisfies the second and third statements of the above Lemma on $\overline{U}$.
\end{rmk}

\begin{lema}\label{lemma_rotation}
	Every self-homeomorphism on $\partial \mb{D}$ which conjugates two irrational rotations is indeed a rotation.
\end{lema}
\begin{proof}
	Suppose $h:\partial \mb{D}\to\partial\mb{D}$ is a homeomorphism as in the statement. Since rotation numbers are preserved under conjugation, we can assume further that
	$$h\circ R_{\theta}=R_{\theta}\circ h\tu{\ \ \ and \ \ \ }h(1)=e^{2\pi i\theta_0}$$ with $\theta\in (0,1)$ an irrational number, $\theta_0\in [0,1)$ and $R_\theta:z\mapsto e^{2\pi i\theta}z$. Then for any $n\geq 0$, we conclude that
	$$h(R_\theta^{\circ n}(1))=R_{\theta}\circ h\circ R_\theta^{\circ n-1}(1)=\cdots=R_\theta^{\circ n}(h(1))=R_\theta^{\circ n}(R_{\theta_0}(1))=R_{\theta_0}(R_\theta^{\circ n}(1)).$$
	
	Claim that the orbit $E:=\cup_{n\geq 0} \{R_\theta^{\circ n}(1)\}$ is dense in $\partial \mb{D}$. Otherwise, one can choose a component $I$ of $\partial \mb{D}\setminus\ol{E}$ with the maximum length. The backward intervals $I, R_{\theta}^{-1}(I), R_{\theta}^{-2}(I),\cdots$ are of the same length and disjoint from $\ol{E}$. But any two of them are either disjoint or exactly coincide, as the length of the union of any two overlap intervals $R_\theta^{-i}(I)$ and $R_\theta^{-j}(I)$ is no less than that of $I$. Since the total length of $\partial \mb{D}$ is finite, there are infinitely many intervals of $\{R_\theta^{-n}(I)\}_{n\geq 0}$ coincide. Suppose $I_0:=R_{\theta}^{-n_1}(I)=R_{\theta}^{-n_2}(I)$ with $1\leq n_1<n_2$. So $R_\theta^{n_2-n_1}(I_0)=I_0$. In particular, $R_\theta^{n_2-n_1}(\partial I_0)=\partial I_0$. This means the irrational rotation $R_\theta^{n_2-n_1}$ has periodic points. It is impossible. Thus the claim holds.
	
	The continuity implies that $h$ is the rotation $R_{\theta_0}$.
	The lemma is complete.
\end{proof}

\section{Quasisymmetric rigidity of carpet Julia sets}\label{quas_rigidity}

\subsection{Carpet Julia sets}
Let $S$ be a (Sierpi\'{n}ski) carpet. Then it is homeomorphic to the standard middle-third carpet. By \emph{Whyburn's characterization}, a set $S$ in $\EC$ is a carpet if and only if it is compact, connected, locally connected, has no local cut-points and has topological dimension one.
A set $E$ is called \emph{buried} in carpet $S$, if $E$ is disjoint from any peripheral circles. A domain $\Omega$ is said to be \emph{clean} for $S$ provided that $\Omega$ is a Jordan domain and the boundary $\partial \Omega$ is buried in $S$. Notice that any buried point in $S$ has arbitrarily small clean neighborhoods in $\EC$.

\begin{lema}\label{chain_domains}
	Let $S$ be a carpet in $\EC$ and $P\subseteq S$ be a finite set. Let $p\neq q$ be two buried points in $S\setminus P$. Then for any
	$\epsilon>0$ there exist clean domains $U_0,\cdots,U_N$ with $p\in U_0, q\in U_N$ such that $P\cap U_k=\emptyset$, $U_k\cap U_{k+1}\neq\emptyset$ and the spherical diameter
	$\text{diam}_{\sigma}U_k<\epsilon$ for $0\leq k\leq N-1$.
\end{lema}

\begin{proof}
	By collapsing the closure of each component of $\EC\setminus S$ to a point, we get the quotient map $\pi:\EC\to\mathbb{S}^2$ by Moore's theorem \cite{Moo25}. Set $\widetilde{p}:=\pi(p), \widetilde{q}:=\pi(q), \widetilde{P}:=\pi(P)$ and $A:=\pi(\EC\setminus S)$. It follows that $A$ is a set of countable many points in $\mathbb{S}^2$. We can choose a simple arc $\widetilde{\gamma}$ in $\mathbb{S}^2\setminus(\widetilde{P}\cup A)$ joining $\widetilde{p}$ and $\widetilde{q}$. Then the lift $\gamma:=\pi^{-1}(\widetilde{\gamma})$ is a simple arc buried in $S\setminus P$ connecting $p$ and $q$.
	
	Note that there exists a local neighborhood basis of any buried point in $S$, which consists of clean domains. Since $\gamma$ is compact, for any $\epsilon>0$, applying a standard argument, one can get a sequence $(U_0, x_0),\cdots,(U_{N},x_N)$ with $U_k\cap U_{k+1}\neq \emptyset, x_k\in U_k\cap \gamma$, $x_0:=p,x_N:=q$ such that the clean domains $U_k$ are disjoint from $P$ and such that $\text{diam}_\sigma(U_k)<\epsilon$ for $0\leq k\leq N$. Thus the Lemma follows.
\end{proof}

\subsection{Rigidity of carpet Julia sets}

The self-homeomorphisms on a carpet form a large group. However, strong rigidity results are valid if one considers the \textit{quasisymmetric} homeomorphism from the carpet onto itself. In \cite{BLM16}, the quasisymmetric rigidity of the carpets that are Julia sets of postcritically-finite rational maps was studied. In this section, we show that the rigidity of carpet Julia sets can be held on a more general situation.

We first recall some definitions. A closed subset of $\EC$ is called a \textit{Schottky set} if its complements are open round disks with disjoint closures. Obviously, a round carpet is a Schottky set. A \textit{relative Schottky set} in a domain $D\subset\EC$ is a subset of $D$ obtained by removing from $D$ a collection of round disks whose closures are contained in $D$ and are pairwise disjoint.

Note that the locally porous property is quasisymmetric invariant. In other words, if $\xi$ is quasiconformal map on $\EC$ sending a locally porous Julia set to a Schottky set $S$, then $S$ is also locally porous. Recall that $\mathcal{F}_f$ and $\mathcal{P}_f$, respectively, denote the Fatou set and the postcritical set of a rational map $f$.

\begin{thm}\label{qs_rigidity}
	Let $f$ be a rational map with carpet Julia set $\mathcal{J}_f$. Suppose that
	
	$\bullet$ $\mathcal{J}_f$ is locally porous and
	
	$\bullet$ $\mathcal{J}_f$ is quasisymmetric equivalent to a round carpet.
	
	Then for any quasisymmetric equivalence $\xi:\mathcal{J}_f\to\mathcal{J}_g$ between carpet Julia sets $\mathcal{J}_f$ and $\mathcal{J}_g$, we have
	
	$(1)$ there exist positive integers $m$, $m'$ and $l$ such that $g^{\circ m}\circ\xi\circ f^{\circ l}=g^{\circ m'}\circ\xi$ on $\mathcal{J}_f$.
	
	$(2)$ furthermore, if both of the sets $\mathcal{F}_f\cap \MP_f $ and $\mathcal{F}_g\cap\MP_g$ are finite, then $\xi$ is the restriction of a M\"{o}bius transformation.
\end{thm}

The proof of this theorem is essentially based on the rigidity of \emph{Schottky maps} on locally porous Schottky sets. We recommend \cite{Mer12, Mer14} for the related theory. Our approach is more or less like the proof of \cite[Theorem 1.4]{BLM16}. But here we need to deal with the Siegel disk situation additionally. Also our proof indicates that if one of the complement component of the carpet Julia set $\mc{J}_f$ is a Siegel disk, then so is the carpet $\mc{J}_g$.

\begin{proof}
	Let $\xi:\mathcal{J}_f\to\mathcal{J}_g$ be a quasisymmetric homeomorphism between $\MJ_f$ and $\MJ_g$.
	Since the peripheral circles of $\mathcal{J}_f$ are uniform quasicircles, by \cite[Proposition 5.1]{Bon11}, there exists a quasiconformal homeomorphism from $\EC$ onto $\EC$ whose restriction on $\mathcal{J}_f$ is $\xi$. We denote this quasiconformal homeomorphism still by $\xi$ for convenience.
	
	(1) Let $p$ be a repelling periodic point buried in $\mathcal{J}_f$ such that $f^{\circ d}(p)=p$ for some $d\geq 1$ (the existence of a such point is obvious since the number of periodic Fatou components are finite but the repelling periodic points are dense in the Julia set). Let $V_1$ be its clean neighborhood (small enough) such that
	the restriction
	$$f_p^{\circ d}:=f^{\circ d}:V_1\to f^{\circ d}(V_1)$$
	is conformal and $\overline{V_1}\subseteq f^{\circ d}(V_1)$.
	
	Let $E$ be a finite subset of $\mathcal{F}_g$ containing at least three points such that $g(E)\cap \xi(V_1)=\emptyset$. Given $k\geq 1$, we know that $g^{\circ m}\circ\xi\circ f_p^{-dk}(V_1)$ will cover the sphere except at most two points for sufficiently large $m$ (see \cite[Theorem 4.10 and Corollary 14.2]{Mil06}). Here $f_p^{-dk}$ denotes the $k$-th inverse of the conformal map $f_p^{\circ d}|_{V_1}$. Then we can set
	$$m(k):=\text{max}\{m;g^{\circ m}\circ\xi\circ f_p^{-dk}(V_1)\cap E=\emptyset \}.$$
	Obviously, $g^{\circ m(k)}\circ\xi\circ f_p^{-dk}(V_1)$ intersects $g^{-1}(E)$ and thus
	\begin{equation}\label{no_constant}
	\text{diam}_\sigma(g^{\circ m(k)}\circ\xi\circ f_p^{-dk}(V_1))\geq \text{dist}_\sigma(g^{-1}(E),\mathcal{J}_g).
	\end{equation}
	So $\{g^{\circ m(k)}\circ\xi\circ f_p^{-dk}|_{V_1}\}_{k\geq 1}$ is a sequence of uniformly quasiregular maps whose images avoid the finite set $E$.
	Applying \cite[Corollary 5.5.7, p.\,182]{AIM09} and composing with a quasiconformal map $\xi^{-1}$, we get a sequence of $K$-quasiregular mappings $\{h_k:=\xi^{-1}\circ g^{\circ m_k}\circ\xi\circ f_p^{-n_k}\}$ which converges locally uniformly to a mapping $h$ on $V_1$, where $m_k:=m(k)$ and $n_k:=dk$. Note that $h$ is also a $K$-quasiregular mapping but not a constant by \eqref{no_constant}.
	
	Since the critical points of a quasiregular mapping are locally finite, we can pick a clean Jordan domain $V_2\Subset V_1$ with $V_2\cap \mathcal{J}_f\neq\emptyset$ such that $h$ on $V_2$ is $K$-quasiconformal. From the \emph{Stoilow fractorization} (see \cite[Corollary 5.5.4, p.\,181]{AIM09}) and the Rouch\'{e}'s Theorem on holomorphic functions, one may assume that
	for sufficiently large $k_0$, the mappings $\{h_k|_{V_2}\}_{k\geq k_0}$ are $K$-quasiconformal and uniformly converges to $h$.
	We also have that
	$$h(V_2\cap \mathcal{J}_f)\subseteq \mathcal{J}_f\text{ and }h(V_2\cap\mathcal{F}_f)\subseteq\mathcal{F}_f.$$
	This is because each $h_k$ sends $V_2\cap\mathcal{J}_f$ and $V_2\cap\mathcal{F}_f$ into $\mathcal{J}_f$ and $\mathcal{F}_f$ respectively.
	
	Again we choose a repelling periodic point $q$ buried in a clean domain $V_3\Subset V_2$ such that $f^{\circ s}(q)=q$ and $f^{\circ s}|_{V_3}$ is conformal.
	Let $\beta$ be a quasisymmetric mapping sending $\mathcal{J}_f$ to a round carpet $S\subseteq \EC$. Its quasiconformal extension
	to the whole sphere is still denoted by $\beta$. If we set $U:=\beta(V_3)$, $\widetilde{q}:=\beta(q)$, $\widetilde{f}=\beta\circ f\circ\beta^{-1}$, $\widetilde{f}_p=\beta\circ f_p\circ\beta^{-1}$, $\widetilde{h}:=\beta\circ h\circ \beta^{-1}|_U$ and $\widetilde{h}_k:=\beta\circ h_k\circ \beta^{-1}|_U$. Then $U\cap S$ is a locally porous relative Schottky set in $U$ and all these mappings $\widetilde{f}$, $\widetilde{h}$, $\widetilde{h_k}$ are Schottky maps in the sense of \cite{Mer14} (see also \cite[Theorem 1.2]{Mer12}).
	
	By \cite[Theorem 4.1]{Mer14}, the derivative of $\widetilde{f}^{\circ s}$ in the sense of \cite{Mer14} at the fixed point $\widetilde{q}$  is not equal to $1$. Then we apply \cite[Theorem 5.2]{Mer14} to have that $\widetilde{h}_k=\widetilde{h}_{k+1}$ on $U\cap S$ for a sufficiently large $k$. That is,
	$$
	\beta\circ\xi^{-1}\circ g^{\circ m_k}\circ\xi\circ f^{-n_k}_p\circ\beta^{-1}=\beta\circ\xi^{-1}\circ g^{\circ m_{k+1}}\circ\xi\circ f_p^{-n_{k+1}}\circ\beta^{-1}
	$$
	on $U\cap S$.
	If we set $\widetilde{g}:=\beta\circ\xi^{-1}\circ g\circ\xi\circ\beta^{-1}$, $m:=m_k$, $l:=n_{k+1}-n_k$ and $m':=m_{k+1}$, then
	\begin{equation}\label{equations}
	\widetilde{g}^{\circ m}\circ\widetilde{f}^{\circ l}=\widetilde{g}^{\circ m'}
	\end{equation}
	hold on $\widetilde{f}_p^{-n_{k+1}}(U\cap S)$.
	
	Now we are going to extend the equation \eqref{equations} to the whole carpet $S$.
	Note that the mappings $\widetilde{g}^{\circ m}\circ\widetilde{f}^{\circ l}$ and $\widetilde{g}^{\circ m'}$ are locally quasisymmetric on $S$ except at the finite set $P:=S\cap(\text{Crit}(\widetilde{g}^{\circ m}\circ\widetilde{f}^{\circ l})\cup\text{Crit}(\widetilde{g}^{\circ m'}))$.
	For any $\widetilde{q}_0$ buried in $S\setminus P$, by Lemma \ref{chain_domains} there exists a sequence of clean domains $U_0:=\widetilde{f}_p^{-n_{k+1}}(U), U_1, \cdots, U_N$ with $\widetilde{q}_0\in U_N$ such that $U_k\cap U_{k+1}\neq\emptyset$, $U_k\cap P=\emptyset$ and on each $U_k$ the mappings $\widetilde{g}^{\circ m}\circ\widetilde{f}^{\circ l}$ and $\widetilde{g}^{\circ m'}$ are injective. Then the restrictions of $\widetilde{g}^{\circ m}\circ\widetilde{f}^{\circ l}$ and $\widetilde{g}^{\circ m'}$ on each $S\cap U_k$ are Schottky maps. According to \cite[Corollary 4.2]{Mer14}, the equation \eqref{equations} holds at $\widetilde{q}_0$. Finally by the arbitrariness of $\widetilde{q}_0$ and the continuity, the equation \eqref{equations} holds also on $P$, on the peripheral circles and thus on the whole $S$.
	By the definition of $\widetilde{g}$ and $\widetilde{f}$, clearly $(1)$ follows.
	
	(2) Suppose $\sharp(\MP_f \cap\mathcal{F}_f)<\infty$ and $\sharp(\MP_g\cap\mathcal{F}_g)<\infty$. Since $\mathcal{J}_{f^{\circ k}}=\mathcal{J}_f$, without loss of generality, we assume the periodic Fatou components are fixed by $f$. In the following argument, we always assume that the integer $l$ in \eqref{equations} is $1$. Otherwise, one can consider the $l$-th iteration of $f$. By Lemma \ref{dynamic_Fatou}, we get two families
	$\{(U, \psi_U)\}_{U\in\text{Comp}(\mathcal{F}_f)}$ and $\{(V, \varphi_V)\}_{V\in\text{Comp}(\mathcal{F}_g)}$.
	
	For each $n\geq 0$, let $U_n$ be a component of $\mathcal{F}_f$ with $n$ the minimal integer such that $f^{\circ n}(U_n)$ is fixed. We also set
	$V_n:=\xi(U_n)$, $W_n:=g^{\circ m}\circ\xi\circ f(U_n)$, $G_n:=\varphi_{W_n}\circ g^{\circ m}\circ\varphi_{V_n}^{-1}$, $\widetilde{\xi}_n:=\varphi_{V_n}\circ\xi\circ \psi_{U_n}^{-1}|_{\partial{\D}}$, $F_{n}:=\psi_{U_{n-1}}\circ f\circ \psi^{-1}_{U_{n}}$ and  $\widetilde{G}_n:=\varphi_{W_n}\circ g^{\circ m'}\circ \varphi^{-1}_{V_n}$, where $U_{-1}:=U_0$.
	
	\vskip 0.2cm
	
	\tb{Claim}\ \  Each $\widetilde{\xi}_n$ is a rotation and thus can be conformally extended to $\mathbb{D}$.
	\begin{proof}[Proof of the claim]
		
		For $n=0$, we have $f(U_0)=U_0$. The $\eqref{equations}$ can be written as
		\begin{equation}\label{start}
		\varphi_{W_0}\circ g^{\circ m}\circ\varphi^{-1}_{V_0}\circ(\varphi_{V_0}\circ\xi\circ\psi^{-1}_{U_0})\circ\psi_{U_0}\circ f\circ\psi_{U_0}^{-1}
		=\varphi_{W_0}\circ g^{\circ m'}\circ\varphi_{V_0}^{-1}\circ(\varphi_{V_0}\circ\xi\circ\psi_{U_0}^{-1})
		\end{equation}
		
		Firstly the Fatou components $U_0, W_0$ cannot be of geometrically attracting type or of parabolic type (since $\#F_f\cap P_f<\infty$ and \cite[Lemma 8.5, Theorem 10.15]{Mil06}).
		
		If $U_0$ is supattracting, then we have $\deg(f|_{\partial U_0})>1$. This implies $\deg(g^{\circ m}|_{V_0})<\deg(g^{\circ m'}|_{V_0})$. Since both $g^{\circ m}$ and $g^{\circ m'}$  carry $V_0$ onto $W_0$ and the fact that periodic Siegel cycles avoid critical points, the Fatou component $W_0$ must be of periodic superattracting type. According to Lemma \ref{dynamic_Fatou} (3) and \cite[Lemma 7.1]{BLM16}, we see that $\widetilde{\xi}_0(z)=\varphi_{V_0}\circ\xi\circ\psi^{-1}_{U_0}(z)=e^{i\theta}z$ is a rational rotation.
		
		If $U_0$ is a Siegel disk, then $\deg(g^{\circ m}|_{V_0})=\deg(g^{\circ m'}|_{V_0})$. There are two cases: $m=m'$ or $m\neq m'$.
		
		If the former one happens, clearly the orbit $V_0\to \cdots \to g^{\circ m-1}(V_0)\to W_0$ may hit Siegel disks or supattracting Fatou components of $g$. However, by Lemma \ref{dynamic_Fatou} (2) (3), we can finally deduce that $P_k\circ\wt{\xi_0}\circ P_{1,\theta}=P_k\circ\wt{\xi_0}$ on $\partial \mb{D}$ for some integer $k\geq 1$ and irrational number $0<\theta<1$ from \eqref{start}. It follows that $\wt{\xi_0}\circ P_{1,\theta}\circ\wt{\xi_0}^{-1}$ is one of the $k$ rational rotations $P_{1, \frac{1}{k}},\cdots,P_{1, \frac{k-1}{k}},P_{1, \frac{k}{k}}$. This is impossible as rotation numbers are invariant under conjugation.
		
		If the latter case happens, that is, $m\neq m'$. Then $W_0$ is a periodic Siegel disk of $g$ by a similar argument as above. The equation \eqref{start} gives
		$$P_k\circ\wt{\xi_0}\circ P_{1,\theta}=P_{1,\eta}\circ P_k\circ\wt{\xi_0}=P_k\circ P_{1,\frac{\eta}{k}}\circ\wt{\xi_0}$$
		for some integer $k\geq 1$ and irrational numbers $0<\theta,\eta<1$ from Lemma \ref{dynamic_Fatou}. Then $\wt{\xi_0}\circ P_{1,\theta}\circ\wt{\xi_0}^{-1}\circ P_{1,\frac{\eta}{k}}^{-1}$ is a rational rotation. It follows that $\wt{\xi_0}$ conjugates two irrational rotations on $\partial \mb{D}$. Then $\wt{\xi_0}$ is a rotation by Lemma \ref{lemma_rotation}.
		
		By induction, for $n=1,2$ and so on, we have
		\begin{equation}\label{lifts}
		H_n:=G_{n-1}\circ\widetilde{\xi}_{n-1}\circ F_{n}=\widetilde{G}_{n}\circ\widetilde{\xi}_n \text{ on }\partial \mathbb{D}.
		\end{equation}
		and so $\deg(H_n)=\deg(\widetilde{G}_{n})$.
		Note that on $\overline{\mathbb{D}}$ the mappings $G_{n-1}$, $\wt{\xi}_{n-1}$, $F_n$, $\widetilde{G}_n$ act as $z\mapsto z^\delta$ or $z\mapsto e^{2\pi i\theta}z$. Then there exists an analytic lift $$\widetilde{\xi}'_n:\overline{\mathbb{D}}\setminus\{0\}\to\overline{\mathbb{D}}\setminus\{0\}$$ of $H_n$ under the unbranched covering $\widetilde{G}_n:\overline{\mathbb{D}}\setminus\{0\}\to \overline{\mathbb{D}}\setminus\{0\}$ by \cite[Proposition 1.33]{Hat02}.
		We may also assume $\widetilde{\xi}'_n(x)=\widetilde{\xi}_n(x)$ for an $x\in \partial{\mathbb{D}}$. From \eqref{lifts}, both $\widetilde{\xi}_n'|_{\partial\mathbb{D}}$ and $\widetilde{\xi}_n$ are of lifts of $H_n|_{\partial\mathbb{D}}$. The unique lifting property \cite[Proposition 1.34]{Hat02} implies that they are the same. Thus we have an analytic extension $\widetilde{\xi}_n=\widetilde{\xi}'_n$ on $\mathbb{D}\setminus\{0\}$ and $\widetilde{\xi}_n(0)=0$. By \eqref{lifts} and induction, we know that $\widetilde{\xi}_n$ is a rotation.
	\end{proof}
	
	We still write $\widetilde{\xi}_n:\overline{\D}\to\overline{\D}$ after extension and modify $\xi$ on each Fatou component as follows,
	\[\xi(z)=\left\{
	\begin{array}{ll}
	\xi(z), & \hbox{if $z\in\mathcal{J}_f$,} \\[5pt]
	\varphi^{-1}_{V_n}\circ\widetilde{\xi}_n\circ\psi_{U_n}(z), & \hbox{if $z\in U_n$ for $U_n\in\text{Comp}(\mathcal{F}_f)$.}
	\end{array}
	\right.
	\]
	It is clear that $\xi$ is a quasiconformal mapping on the sphere and is conformal except on the Julia set $\mathcal{J}_f$. The locally porous property implies that $\mathcal{J}_f$ has measure zero. Thus $\xi$ is a M\"{o}bius transformation. The proof is complete.
\end{proof}

\noindent{\textit{Proof of Theorem \ref{mobius}}}. Suppose that $f$ is a semi-hyperbolic rational map whose Julia set is a carpet, then $\MJ_f$ is locally porous by Theorem \ref{local_porous}. If further the $\omega$-limit sets of the critical points of $f$ are disjoint from the elements of the peripheral circles of $\MJ_f$, then $\MJ_f$ is quasisymmetrically equivalent to a round carpet by Theorems \ref{thm-main-1} and \ref{thm-main-2}. According to Theorem \ref{qs_rigidity}, if $\sharp(\MF_f\cap\MP_f)<\infty$ and $\sharp(\MF_g\cap\MP_g)<\infty$, then any quasisymmetric homeomorphism $\xi:\MJ_f\to\MJ_g$ is the restriction of a M\"{o}bius transformation.
\hfill $\square$

\subsection{The group of quasisymmetries is finite}

For a given rational map $f$, we use $QS(\mathcal{J}_f)$ to denote the set of all the quasisymmetric homeomorphims that map $\MJ_f$ onto itself. It is easy to see that $QS(\mathcal{J}_f)$ forms a group which is nonempty since the identity belongs to $QS(\mathcal{J}_f)$.

\begin{thm}\label{finite_group}
	Let $f$ be the rational map as stated in Theorem \ref{qs_rigidity}. Then the quasisymmetric group $QS(\mathcal{J}_f)$ is finite.
\end{thm}
This theorem generalizes the result of \cite[Corollary 1.2]{BLM16}. Our strategy in the proof here differs from \cite[Corollary 1.2]{BLM16} on handling the situation that $\xi$ is a parabolic transformation.

\begin{proof}
	Without loss of generality, we assume that $\sharp(\MP_f \cap \mathcal{F}_f)$ is finite. Otherwise, we can do a surgery on the Fatou set such that $f$ is quasiconformally conjugate to a rational map, whose postcritical set is finite on the Fatou set, between their Julia sets (see \cite[\S 7]{BF14}). By Theorem \ref{qs_rigidity}, we know that $QS(\mathcal{J}_f)$ consists of the restriction of M\"{o}bius transformations.
	
	Firstly we claim that the group $QS(\mathcal{J}_f)$ is \emph{discrete}, i.e.,
	there exists $\delta>0$ such that
	$$\tu{inf}_{h\in QS(\mathcal{J}_f)\sm\{\tu{id}_{\mathcal{J}_f}\}}\left(\tu{max}_{z\in \mathcal{J}_f}\dist_\sigma(h(z),z)\right)\geq\delta.$$
	If not, there exists a sequence of distinct elements $\{h_k\}_{k\geq 1}\subseteq QS(\MJ_f)$ converging to $\tu{id}_{\mathcal{J}_f}$. Let $C_1,C_2$ and $C_3$ be three different peripheral circles. Then the Hausdorff distances between $C_i$ and $h_k(C_i)$ tends to zero as $k$ tends to $\infty$. Since all $h_k(C_i)$ are either disjoint or coincides for $k\geq 1$ and $1\leq i\leq 3$. It follows that $h_k(C_i)=C_i$ for sufficiently large $k$. Note that for loxodromic or parabolic transformation $g$, the sequence $\{g^{\circ k}\}_{k\geq 1}$ converges locally uniformly to a constant on $\EC\setminus\text{Fix}(g)$, where $\text{Fix}(g)$ is the set of all the fixed points of $g$. Thus $h_k$ can only be elliptic (for the classification of M\"{o}bius transformation, see \cite[p.\,67]{Bea83}). However, an elliptic element can not keep three disjoint Jordan domains fixed. Thus $h_k=\text{id}$ for sufficiently large $k$, which is a contradiction. The claim follows.
	
	Now letting $\xi\neq\tu{id}_{\mathcal{J}_f}$ in $QS(\mathcal{J}_f)$, we claim that $\xi$ cannot be loxodromic and parabolic. Indeed, otherwise, let $p$ be a fixed point of $\xi$ (repelling fixed point in the loxodromic case). Since $\xi$ keeps $\mathcal{J}_f$ invariant, it follows that $p$ is buried in $\mathcal{J}_f$. Let $U_1$ be a small clean neighborhood of $p$. Choose a clean domain $U_2\subseteq U_1$ with $U_2\cap\mathcal{J}_f\neq \emptyset$ such that
	$$\xi^{-k}(U_2)\subseteq U_1, k\geq 0\text{ and }\xi^{-k}(U_2)\to p\text{ as }k\to\infty.$$
	Indeed, such $U_2$ exists obviously in the loxodromic case. Since a parabolic M\"{o}bius transformation has a repelling petal near $p$, hence a clean domain in the petal is as required.
	
	In the following, we apply the same argument as the proof of Theorem \ref{qs_rigidity}. Let $E$ be a finite set of $\mathcal{F}_f$ containing at least three points such that $E\cap U_1=\emptyset$. Given $k\geq 1$, we set
	$$m(k):=\text{max}\{m;f^{\circ m}\circ\xi^{-k}(U_2)\cap E=\emptyset\}.$$
	Then as in proof of Theorem \ref{qs_rigidity}, the sequence of holomorphic mappings $\{f^{\circ m(k)}\circ\xi^{-k}|_{U_2}\}_{k\geq 1}$ is a normal family.
	Then there exists a subsequence $\{h_k:=f^{\circ m_k}\circ\xi^{-n_k}|_{U_2}\}$ converging locally uniformly to a non-constant holomorphic mapping $h$. A similar argument as in Theorem \ref{qs_rigidity} shows that
	$h_k=h_{k+1}$ on $U_2$ for sufficiently large $k$. Thus
	$$f^{\circ m}=f^{\circ m'}\circ\xi^{\circ l}\text{ on }U\cap \mathcal{J}_f$$
	for some $m\geq 1, m'\geq 1,l\geq 1$ and clean domain $U$. By the local uniqueness of holomorphic mappings, we have $f^{\circ m}=f^{\circ m'}\circ \xi^{\circ l}$ on the whole sphere and so $m=m'$ (by considering the degrees).
	
	Take a point $q\in\EC$ such that $f^{-m}(q)$ does not contain any fixed point of $\xi^{\circ l}$. Since $f^{\circ m}=f^{\circ m}\circ \xi^{\circ l}$, it follows that $\xi^{\circ l}(f^{-m}(q))\subseteq f^{-m}(q)$, which is impossible if $\xi$ is loxodromic or parabolic.
	
	Thus $QS(\mathcal{J}_f)$ consists of only elliptic transformations. By \cite[Theorem 4.3.7]{Bea83}, elements of $QS(\mathcal{J}_f)$ share a common fixed point. Since $QS(\mathcal{J}_f)$ is discrete, it follows that $QS(\mathcal{J}_f)$ is finite.
\end{proof}

\vskip0.2cm
\noindent{\textit{Proof of Theorem \ref{group_finite}}}. Suppose that $f$ is a semi-hyperbolic rational map whose Julia set is a carpet and the $\omega$-limit sets of the critical points of $f$ are disjoint from the elements of the peripheral circles of $\MJ_f$, then $\MJ_f$ is locally porous and quasisymmetrically equivalent to a round carpet by Theorems \ref{local_porous}, \ref{thm-main-1} and \ref{thm-main-2}. Therefore, $QS(\MJ_f)$ is finite according to Theorem \ref{finite_group}.
\hfill $\square$

\section{An example of postcritically-infinite carpet Julia set}\label{example}

In this section, we will construct a carpet Julia set of a rational map such that it is quasisymmetrically equivalent to a round carpet. However, the rational map $f$ is semi-hyperbolic and has an infinite critical orbit in $\MJ_f$.

Let $q:\mathbb{R}/\mathbb{Z}\to\mathbb{R}/\mathbb{Z}$ be the \textit{doubling map} defined by $q(t)=2t\,\textup{mod}\, \mathbb{Z}$ and
\begin{equation}
l(t):=
\left\{                         
\begin{array}{ll}               
2t &~~~~\text{if}~0\leq t<1/2, \\             
2-2t &~~~~\text{if}~1/2\leq t<1.    
\end{array}                     
\right.                         
\end{equation}
be the length of the component $(\mathbb{R}/\mathbb{Z})\setminus\{t,1-t\}$ containing 0. Let $T(t)=\textup{min}\{2t,2-2t\}$ be the \textit{tent map} on the interval $[0,1]$. One can easily check that
\begin{equation}\label{communi}
T\circ l(t)=l\circ q(t)
\end{equation}
for all $t\in[0,1]$. Actually, the map $l(t)$ is equal to $T(t)$. We use these notations here by following Tiozzo's paper \cite[p.\,675]{Tio15}.

\begin{lema}\label{rational}
	Let $0\leq \alpha\leq1$ be a real number. Then $\alpha$ is rational if and only if $\alpha$ is (pre-)periodic under the iteration of the doubling map $q$.
\end{lema}

\begin{proof}
	Obviously, this lemma holds for $\alpha=0$ or $1$. Hence we assume that $0<\alpha<1$. If $\alpha$ is (pre-)periodic, then there exist two different integers $k_1,k_2\geq 0$ such that $q^{\circ k_1}(\alpha)\equiv q^{\circ k_2}(\alpha)\,\textup{mod}\, \mathbb{Z}$. This means that there exists an integer $k_3$ such that $2^{k_1}\alpha=2^{k_2}\alpha+k_3$. Then $\alpha=k_3/(2^{k_1}-2^{k_2})$ is a rational number.
	
	Conversely, we only need to prove that, if $\alpha=m/n$ is a rational number with the simplest expression, where $n$ is odd, then $\alpha$ is periodic under $q$. Consider the restriction of $q$ on the set $S:=\{0,1/n,\cdots,(n-1)/n\}$:
	\begin{equation*}
	h:=q|_S\,:\,\frac{t}{n}\mapsto \frac{2\,t\,\textup{mod } n}{n}.
	\end{equation*}
	We claim that $h$ is injective. Indeed, if $h(t_1/n)=h(t_2/n)$, then $2(t_1-t_2)=k\,n$ holds for some integer $k$. Since $n$ is odd, it follows that $k$ is even and $|t_1-t_2|=|\frac{k}{2}|\cdot n\leq n-1$. This means that $k=0$ and $t_1=t_2$. The finiteness of the cardinal number of $S$ implies every element in $S$ is pre-periodic under $h$. Then each element in $S$ is periodic. Otherwise, there will be at least two elements which are mapped to the same element. This contradicts with that $h$ is a injection.
\end{proof}

In the following, based on the combinatorial theory of quadratic polynomials and renormalization theory,  we shall construct a critically-infinite semi-hyperbolic McMullen map whose Julia set is quasisymmetrically equivalent to a round carpet.

\begin{thm}\label{existence-mcm}
	There exists a suitable parameter $\lambda>0$ such that the McMullen map
	\begin{equation}
	f_\lambda(z)=z^d+\lambda/z^d
	\end{equation}
	is semi-hyperbolic, the critical orbit in the Julia set is infinite and the corresponding Julia set is quasisymmetrically equivalent to a round Sierpi\'{n}ski carpet, where $d\geq 3$.
\end{thm}

\begin{proof}
	We divide the construction into three main steps as following.
	
	\textbf{Step 1.}
	For a given irrational number $\alpha\in(0,1)$, one can write it as an infinite binary sequence $\alpha=0.a_1a_2a_3\cdots$ by Lemma $\ref{rational}$, where $a_i\in\{0,1\}$. Define a binary number
	\begin{equation*}
	\theta:=0.0\underbrace{1\cdots 1}_{100}\underbrace{0\cdots0}_{b_1}\underbrace{1\cdots1}_{b_2}\underbrace{0\cdots0}_{b_3}\cdots,
	\end{equation*}
	where $b_i=a_i+1$ for $i\geq 1$. Then $1\leq b_i\leq 2$ and we have:
	
	$\bullet$ The number $\theta\in(0,1/2)$ is irrational. If not, by Lemma $\ref{rational}$, the number $\theta$ will be eventually periodic under the iteration of the doubling map $q$. Then there exist $m\geq 2$ and $p\geq 2$ such that
	\begin{equation*}
	q^{\circ n}(\theta)=0.\,\overline{\underbrace{1\cdots 1}_{b_m}\underbrace{0\cdots 0}_{b_{m+1}}
		\,\cdots\,\underbrace{1\cdots 1}_{b_{m+p-2}}\underbrace{0\cdots 0}_{b_{m+p-1}}},
	\end{equation*}
	where $n=101+b_1+\cdots+b_{m-1}$. This means that the sequence $b_m$, $\cdots$, $b_{m+p-1}$, $b_{m+p}$, $\cdots$, $b_{m+2p-1}$, $\cdots$, is periodic with period $p$. Therefore, the sequence $a_m$, $\cdots$, $a_{m+p-1}$, $a_{m+p}$, $\cdots$, $a_{m+2p-1}$, $\cdots$ is also periodic with period $p$ since $a_i=b_i-1$ for each $i$. Then $\alpha=0.a_1\cdots a_{m-1}\overline{a_m a_{m+1}\cdots a_{m+p-1}}$ is a rational number by Lemma \ref{rational}. This is a contradiction since $\alpha$ is irrational.
	
	$\bullet$ Define a rational number with the binary form
	\begin{equation*}
	\theta':=0.\overline{0\underbrace{1\cdots1}_{99}}.
	\end{equation*}
	Then $0<\theta'<\theta<1/2$ and $\theta',\theta$ are very close to $1/2$. We have
	\begin{equation*}
	l(\theta')=0.\underbrace{1\cdots1}_{99}\overline{0\underbrace{1\cdots1}_{99}} \text{~and~}
	l(\theta)=0.\underbrace{1\cdots 1}_{100}\underbrace{0\cdots0}_{b_1}\underbrace{1\cdots1}_{b_2}\underbrace{0\cdots0}_{b_3}\cdots.
	\end{equation*}
	For any $n\geq 2$, by a direct calculation, it is easy to check that
	\begin{equation}\label{inequl}
	0<l(q(\theta))<l(q(\theta'))<l(q^{\circ n}(\theta)),\,l(q^{\circ n}(\theta'))<l(\theta')<l(\theta).
	\end{equation}
	
	$\bullet$ Define a set
	\begin{equation}\label{defi-R}
	\mathcal{R}:=\{t\in\mathbb{R}/\mathbb{Z}: T^{\circ n}(l(t))\leq l(t) \text{~for all~} n\geq 0\}.
	\end{equation}
	By \eqref{communi} and \eqref{inequl}, we have $\theta',\,\theta\in\mathcal{R}$.
	\vskip 0.28cm
	\textbf{Step 2.}
	Construct a quadratic polynomial $P_c(z)=z^2+c$ with the following properties:
	
	(1) The critical orbit $\MO^+_{P_c}(0)=\{P_c^{\circ n}(0):n\geq 0\}$ is contained in the Julia set of $P_c$ and the cardinal number of $\MO^+_{P_c}(0)$ is infinite.
	
	(2) The critical point $0$ is non-recurrent and the $\omega$-limit set of $0$ does not contain the $\beta$ \emph{fixed point} of $P_c$. Recall that a $\beta$ fixed point of a polynomial is the landing point of \emph{dynamical external ray} with angle zero.
	
	The set $\mathcal{R}$ defined in \eqref{defi-R} is exactly the set of all angles of parameter rays whose \emph{prime-end impression} intersects the subset $\mathbb{R}\cap M=[-2,1/4]$ of the Mandelbrot set $M$ (see \cite[Proposition 8.4, p.\,677]{Tio15}). By \cite[Theorem 3.3]{Zak03}, there exists a real number $c:=c(\theta)\in[-2,1/4]$ in the boundary of the Mandelbrot set such that $c$ is contained in the prime-end impression of the parameter rays $R_M(\pm\theta)$ since $\theta\in\mathcal{R}$. Moreover, on the dynamical plane, the dynamical rays $R_c(\pm\theta)$ land at the critical value $c$ of $P_c(z)=z^2+c$.
	
	In fact, such $c$ is unique. Otherwise, suppose that there exists another $c'\neq c$, such that $c'$ is contained in the prime-end impression of the parameter rays $R_M(\pm\theta)$. By the density of hyperbolic parameters in $\mathbb{R}\cap M$ (see \cite{GS97} and \cite{Lyu97}), there is a real hyperbolic parameter $\widetilde{c}$ between $c$ and $c'$ with a pair of rational parameter rays landing at it. This means that $c'$ and $c$ cannot lie in the same prime-end impression of $R_M(\pm\theta)$ at the same time, which is a contradiction.
	
	Now we prove that the quadratic polynomial $P_c$ is the desired map. Again by \cite[Proposition 8.4]{Tio15}, the parameter rays $R_M(\pm\theta')$ land at a parabolic parameter $c_0\in\mathbb{R}$ since $\theta'\in\mathcal{R}$ is a rational number. These two rays together with their landing point $R_M(\theta')\cup R_M(-\theta')\cup\{c_0\}$ bounds a \emph{wake} $W\ni\{-2\}$ with the following property: The quadratic map $P_{\xi}(z)=z^2+\xi$ has a repelling periodic point with exactly two dynamical rays $R_{\xi}(\pm\theta')$ landing at if and only if $\xi\in W$ (see \cite[Theorem 1.2]{Mil00}). By the construction in Step 1, we have $0<\theta'<\theta<1/2$. Then $R_M(\pm\theta)\cup\{c\}\subseteq W$ and hence $R_c(\pm\theta')$ land at a repelling periodic point of $P_{c}$ on the real line. Also, the image $R_c(\pm q(\theta'))$ of $R_c(\pm\theta')$ land at some point on the real line.
	
	Denote by $H$ the simply connected domain bounding by the four dynamical rays $R_c(\pm\theta')$ and $R_c(\pm q(\theta'))$. The two dynamical rays $R_c(\theta)$ and $R_c(0)$ are contained in different components of $\mathbb{C}\setminus\overline{H}$. Moreover, all the dynamical rays $R_c(\pm q^n(\theta))$, where $n\geq 2$, are contained in $H$ by the definition of $\mathcal{R}$ and $\theta'$. This means that the collection of their landing points $\bigcup_{n\geq2}P_c^{\circ n}(c)$ are contained in $H$. Therefore, the critical value $c$ (which is the landing point of $R_c(\pm\theta)$) and the $\beta$ fixed point of $P_c$ are not contained in the $\omega$-limit set of the origin.
	
	\textbf{Step 3.} Construct the semi-hyperbolic rational map $f_\lambda$ whose Julia set is quasisymmetrically equivalent to a round carpet. Consider the McMullen map $f_\lambda(z)=z^d+\lambda/z^d$, where $\lambda\in\C\setminus\{0\}$ and $d\geq 3$. The \textit{free} critical points of $f_\lambda$ are $2d$-th unit roots of $\lambda$. They are either escaping to $\infty$ or have bounded orbits at the same time. The \emph{non-escaping locus} of $f_\lambda$ is defined as
	\begin{equation*}
	\Lambda_d:=\{\lambda\in\mathbb{C}\setminus\{0\}: \textup{The free critical orbits of }f_\lambda \textup{ are not attracted by } \infty \}.
	\end{equation*}
	See left picture in Figure \ref{parameter} for the non-escaping locus of $f_\lambda$ when $d=3$.
	
	\begin{figure}[!htpb]
		\setlength{\unitlength}{1mm}
		\centering
		\includegraphics[width=60mm]{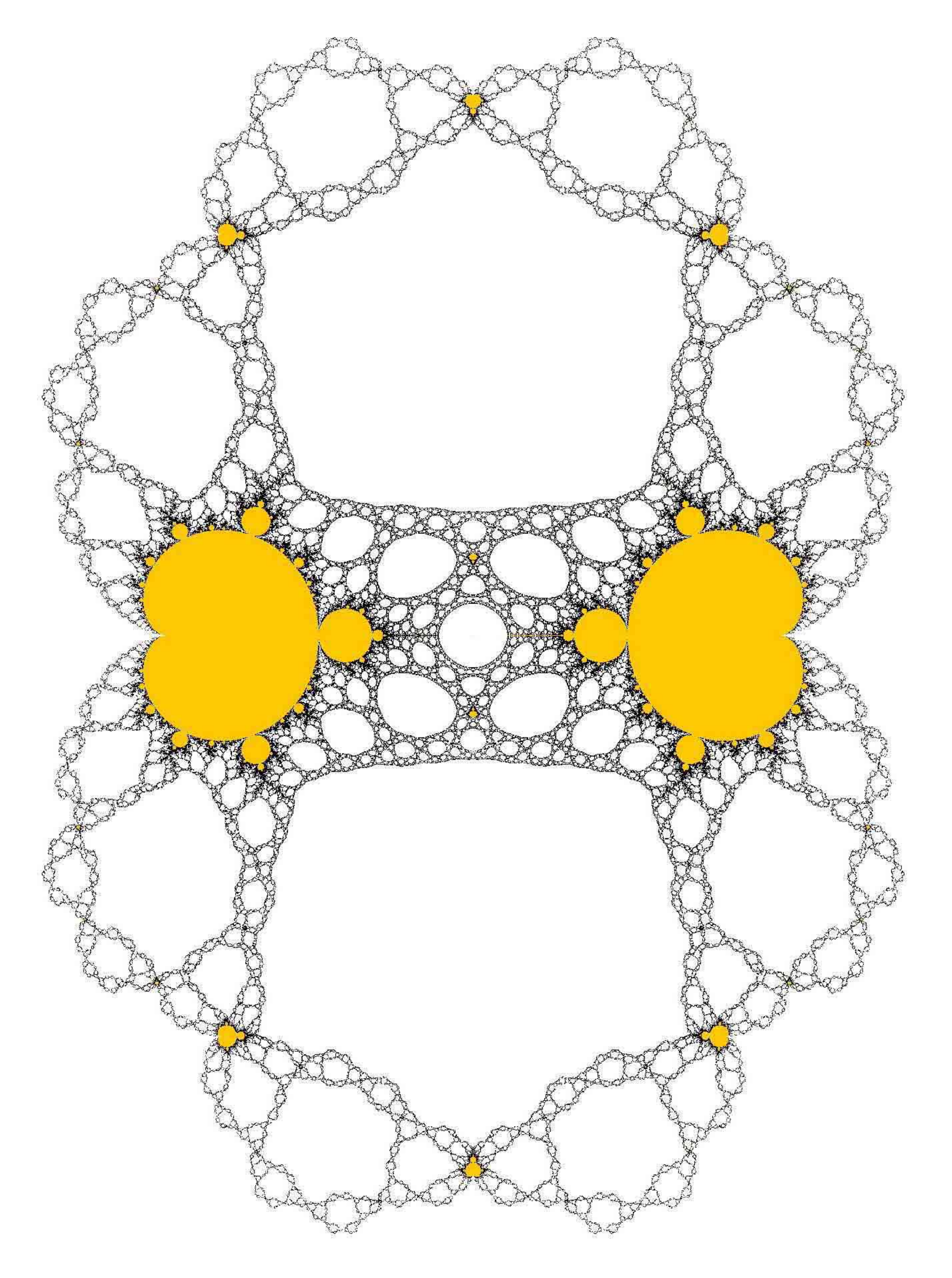}
		\caption{The non-escaping locus $\Lambda_3$ of $f_\lambda$ (the left picture) contains infinitely many homeomorphic copies of the Mandelbrot set.}
		\label{parameter}
	\end{figure}
	
	According to \cite[Theorem 9, p.\,178]{Ste06}, there exists exactly one copy $\mathcal{M}$ of the Mandelbrot set of \textit{order} one in $\Lambda_d\cap\{\lambda\in\C^*:|\arg(\lambda)|<\pi/(d-1)\}$ (Note that there exists a semiconjugacy between $f_\lambda$ and the rational map discussed in \cite[Theorem 9]{Ste06}). The copy $\mathcal{M}$ is symmetric with respect to the positive real axis. Moreover, there exists a homeomorphism $\Phi:\mathcal{M}\to M$ such that, for every $\lambda\in\mathbb{R}^{+}\cap \mathcal{M}=\R^+\cap\Lambda_d$, there is a corresponding parameter $\Phi(\lambda)\in[-2,1/4]$ and the Julia set $\MJ_{f_\lambda}$ contains an embedded set $\tilde{\MJ}_{P_{\Phi(\lambda)}}$, which is homeomorphic to the  Julia set of the quadratic polynomial $P_{\Phi(\lambda)}(z)=z^2+\Phi(\lambda)$. Moreover, the restriction of $f_\lambda$ in a neighborhood of $\tilde{\MJ}_{P_{\Phi(\lambda)}}$ is quasiconformally conjugated to the restriction of $P_{\Phi(\lambda)}$ in a neighborhood of $\MJ_{P_{\Phi(\lambda)}}$.
	
	Let $\lambda_0=\Phi^{-1}(c)\in\mathcal{M}\cap\R^+$, where $c=c(\theta)\in (-2,1/4)$ is the real parameter on the boundary of the Mandelbrot set determined in Step 2. By the symmetry of McMullen maps, all $2d$ free critical points of $f_{\lambda_0}$ are non-recurrent and they have infinite forward orbits. This means that $f_{\lambda_0}$ is semi-hyperbolic (and not sub-hyperbolic). Let $B_\infty$ be the immediate attracting basin of $\infty$ of $f_{\lambda_0}$. Then $\tilde{\MJ}_{P_{\Phi(\lambda_0)}}\cap B_\infty=\{z_{\lambda_0}\}$, where $z_{\lambda_0}$ is the image of the $\beta$ fixed point of $\MJ_{P_c}$ under the quasiconformal conjugacy stated above \cite[Lemma 4.1]{QXY12}. Note that $B_\infty$ is the unique periodic Fatou component of $f_{\lambda_0}$, it follows that the $\omega$-limit sets of the critical points of $f_{\lambda_0}$ are disjoint from the periodic Fatou component of $f_{\lambda_0}$ by the construction of $P_c$.
	
	By \cite[Lemma 4.4]{QXY12}, the Julia set of $f_{\lambda_0}$ is a Sierpi\'{n}ski carpet. By Theorem \ref{thm-main-2}, the peripheral circles of $\MJ_{f_{\lambda_0}}$ are uniform quasicircles and uniformly relatively separated. By Bonk's criterion (\cite[Corollary 1.2]{Bon11}), the Julia set of $f_{\lambda_0}$ is quasisymmetrically equivalent to a round carpet, as desired.
\end{proof}

\begin{rmk}
	One can refer \cite{GZ15} for the study of the generalization of the first and second steps in the proof of Theorem \ref{existence-mcm}.
\end{rmk}

\noindent\textbf{Acknowledgements.}
We would like to thank Yin Yongcheng for useful discussions and the referee for the careful reading and useful suggestions. The first author was supported by the NSFC under grant Nos. 11271074, 11671091 and 11731003. The second author was supported by the NSFC under grant Nos. 11401298 and 11671092. The third author was supported by the NSFC under grant No. 11471317.


\end{document}